\documentclass{article}

     \PassOptionsToPackage{numbers, compress}{natbib}

     \usepackage[preprint]{neurips_2020}

\usepackage[utf8]{inputenc} 
\usepackage[T1]{fontenc}    
\usepackage{hyperref}       
\usepackage{url}            
\usepackage{booktabs}       
\usepackage{amsfonts}       
\usepackage{nicefrac}       
\usepackage{microtype}      
\usepackage{amsmath}
\usepackage{amsthm}
\usepackage{xcolor}
\usepackage[pdftex]{graphicx}
\usepackage{floatrow}
\usepackage{multirow}
\usepackage{longtable}
\usepackage{subcaption}

\newtheorem{thm}{Theorem}
\newtheorem{cor}[thm]{Corollary}
\newtheorem{pro}[thm]{Proposition}

\newcommand{\so}{\mathrm{SO}}
\newcommand{\sym}{\mathrm{Sym}}
\newcommand{\RQ}{\mathbf{R}_Q}
\newcommand{\RQQ}{\mathbf{R}_{QQ}}

\definecolor{clr0}{rgb}{0.9, 0.7, 0}
\definecolor{clr1}{rgb}{0, 0, 1}
\definecolor{clr2}{rgb}{0, 0.7, 0}
\definecolor{clr3}{rgb}{1, 0, 0}
\definecolor{clr4}{rgb}{0.75, 0, 1}
\definecolor{clr5}{rgb}{0.7, 0.4, 0}
\definecolor{clr6}{rgb}{0.3, 0.3, 0.3}

\newcommand{\cb}[1]{\textcolor{#1}{\textbullet}}

\newtheorem{innercustomthm}{Theorem}
\newenvironment{customthm}[1]
  {\renewcommand\theinnercustomthm{#1}\innercustomthm}
  {\endinnercustomthm}
  
\newtheorem{innercustomprop}{Proposition}
\newenvironment{customprop}[1]
  {\renewcommand\theinnercustomprop{#1}\innercustomprop}
  {\endinnercustomprop}

\newcommand{\atan}{\mathrm{atan2}}
\newcommand{\atanl}{\mathrm{atan2}_{\triangleleft}}
\newcommand{\atanr}{\mathrm{atan2}_{\triangleright}}

\newcommand{\tp}{\hspace{-2pt}+\hspace{-2pt}}
\newcommand{\tm}{\hspace{-2pt}-\hspace{-2pt}}

\title{Revisiting the Continuity of Rotation \\ Representations in Neural Networks}

\author{%
  Sitao Xiang \\
  University of Southern California \\
  \texttt{sitaoxia@usc.edu} \\
  \And
  Hao Li \\
  Pinscreen, Inc. \\
  \texttt{hao@hao-li.com} \\
}

\begin{document}

\maketitle

\begin{abstract}
  In this paper, we provide some careful analysis of certain pathological behavior
  of Euler angles and unit quaternions encountered in previous works related
  to rotation representation in neural networks. In particular, we show that for certain problems,
  these two representations will provably produce completely wrong results for some inputs, and
  that this behavior is inherent in the topological property of the problem itself and is not
  caused by unsuitable network architectures or training procedures.
  We further show that previously proposed embeddings
  of $\so(3)$ into higher dimensional Euclidean spaces aimed at fixing this behavior
  are not universally effective, due to possible symmetry in the input causing changes
  to the topology of the input space. We propose an ensemble trick as an alternative solution.
  
\end{abstract}

\section{Introduction}

Quaternions and Euler angles have traditionally been used to represent 3D rotations in computer graphics and vision. This tradition is preserved in more recent works where neural networks are employed for inferring or synthesizing rotations, for a wide range of applications such as pose estimation from images, e.g. \cite{xiang2017posecnn}, and skeleton motion synthesis, e.g. \cite{villegas2018neural}. However, difficulties has been encountered, in that the network seems unable to avoid rotation estimation errors in excess of $100^\circ$ in certain cases, as reported by \cite{xiang2017posecnn}. Attempts has been made to explain this, including arguments that Euler angle and quaternion representations are not embeddings and in a certain sense discontinuous \cite{zhou2019continuity}, and from symmetry present in the data \cite{xiang2017posecnn,saxena2009learning}. One proposed solution is to use embeddings of $\so(3)$ into $\mathbb{R}^5$ or $\mathbb{R}^6$ \cite{zhou2019continuity}. However, we feel that these arguments are mostly based on intuition and empirical results from experiments, while the nature of the problem is topological which is one aspect that has not been examined in depth. In this paper we aim to give a more precise characterization of this problem, theoretically prove the existence of high errors, analysis the effect of symmetries, and propose a solution to this problem. In particular:
\begin{itemize}
    \item We prove that a neural network converting rotation matrices to quaternions and Euler angles must produce an error of $180^\circ$ for some input.
    \item We prove that symmetries in the input cause embeddings to also produce high errors, and calculate error bounds for each kind of symmetry.
    \item We propose the self-selected ensemble, a method that works well with many different rotation representations, even in the presense of input symmetry.
\end{itemize}
We further verify our theoretical claims with experiments.

\section{Theoretical Results}

\subsection{Guaranteed Occurrence of High Errors}

We first consider a toy problem: given a 3-d rotation represented by a rotation matrix, we want to convert it to other rotation representations with neural networks. We will see that under the very weak assumption that our neural network computes a continuous function, it is provable that given any such network that converts rotation matrices to quaternions or Eular angles, there always exists inputs on witch the network produces outputs with high error.

When treating quaternions as Euclidean vectors, we identify $\mathbf{q}=a+b\mathbf{i}+c\mathbf{j}+d\mathbf{k}$ with $(a,b,c,d)$.
We denote the vector dot product between $\mathbf{p}$ and $\mathbf{q}$ with a dot as $\mathbf{p}\cdot \mathbf{q}$ and quaternion multiplication with juxtaposition as $\mathbf{p}\mathbf{q}$. The quaternion conjugate of $\mathbf{q}$ is $\overline{\mathbf{q}}$ and the norm of $\mathbf{q}$ which is the same for quaternions and vectors is $||\mathbf{q}||$.

It has been noticed that any function that converts 3D rotation matrices to their corresponding quaternion exhibits some ``discontinuities'' and that this is related to the fact that $\so(3)$ does not embed in $\mathbb{R}^4$. This has been argued by giving a specific conversion function $f$ and finding discontinuities. Most often, given a rotation matrix $M$, if $\mathrm{tr}(M)>-1$ we have
\begin{equation}
    f(M)=(\frac{t}{2}, \frac{1}{2t}(M_{32}-M_{23}), \frac{1}{2t}(M_{13}-M_{31}), \frac{1}{2t}(M_{21}-M_{12}))
\end{equation}
where $t=\sqrt{1+\mathrm{tr}(M)}$. Since quaternions $\mathbf{q}$ and $-\mathbf{q}$ give the same rotation, any conversion from rotation matrix to quaternion needs to break ties. The conversion given above breaks ties towards the first coordinate being positive. When it equals zero there needs to be additional rules that are not relevant here. Then discontinuities can be found by taking limits on the ``decision boundary'': consider $r(t): [0, 1] \to \so(3)$ defined by
\begin{equation}
\label{eqn:rt}
r(t)=
\begin{bmatrix}
    \cos 2\pi t & -\sin 2\pi t & 0 \\
    \sin 2\pi t & \cos 2\pi t  & 0 \\
    0           & 0            & 1 \\
\end{bmatrix}
\end{equation}
That is, $r(t)$ is the rotation around $z$-axis by angle $2\pi t$. 
Then $f(r(t))=(\cos \pi t, 0, 0, \sin \pi t)$ when $r\in[0,\frac{1}{2})$ and $f(r(t))=(-\cos \pi t, 0, 0, -\sin \pi t)$ when $r\in(\frac{1}{2},1]$. So, we have
\begin{equation}
\lim_{t\to\frac{1}{2}^-}f(r(t))=(0,0,0,1)\neq(0,0,0,-1)=\lim_{t\to\frac{1}{2}^+}f(r(t))
\end{equation}
Thus $f$ is not continuous at $r(\frac{1}{2})$. Since neural networks typically compute continuous functions, such a function cannot be computed by a neural network.

However, we feel that this argument fails to address this problem satisfactorily. Firstly, it pertains to a specific conversion rule. The ties can be broken towards any hemisphere, for which there are an infinite number of choices. In fact the image of $f$ needs not be a hemisphere. In addition, there is no reason to mandate a specific conversion function for the neural network to fit. We need to prove that even with the freedom of learning its own tie-breaking rules, the neural network cannot learn a correct conversion from rotation matrices to quaternions.

Another shortcoming is that this argument does not give error bounds. Since neural networks can only approximate the conversion anyways, can we get a continuous conversion function if we allow some errors and if so, how large does the margin have to be? Experiments in \cite{zhou2019continuity} hint at an unavoidable maximum error of $180^\circ$, which is the largest possible distance between two 3D rotations. We want to prove that. Now we introduce our first theorem. Let $\RQ:S^3 \to \so(3)$ be the standard conversion from a quaternion to the rotation it represents:
\begin{equation}
\RQ(w, x, y, z)=
\begin{bmatrix}
    1-2y^2-2z^2 & 2(xy-zw)    & 2(xz+yw)  \\
    2(xy+zw)    & 1-2x^2-2z^2 & 2(yz-xw)  \\
    2(xz-yw)    & 2(yz+xw)    & 1-2x^2-2y^2 \\
\end{bmatrix}
\end{equation}
and let $d(R_1,R_2)$ denote the distance between two rotations $R_1$ and $R_2$, measured as an angle. To reduce cumbersome notations we overload $d$ so that when a quaternion $\mathbf{q}$ appear as an argument in $d$ we mean $\RQ(\mathbf{q})$. We have $d(\mathbf{p},\mathbf{q})=2\cos^{-1}|\mathbf{p}\cdot \mathbf{q}|$ (see appendix \ref{sec:math}).

\begin{thm}
\label{thm:1func}
For any continuous function $f: \so(3) \to S^3$, there exists a rotation $R\in \so(3)$ such that $d(R,f(R))=\pi$.
\end{thm}
\begin{proof}
Let $r(t): [0, 1] \to \so(3)$ be defined as in equation \ref{eqn:rt}.
$r(t)$ is continuous in $t$. Consider $r$ as a path in $\so(3)$. $r(0)=r(1)=I_{3\times 3}$, so $r$ is a loop.

$S^3$ is a covering space of $\so(3)$ with $\RQ$ being the covering map. $\RQ(1, 0, 0, 0)=I_{3\times 3}=r(0)$. By the lifting property of covering spaces,\footnote{See standard algebraic topology texts, e.g. page 60 of \cite{hatcher2001algebraic}} $r$ lifts to a unique path in $S^3$ starting from $(1, 0, 0, 0)$. That is, there exists a unique continuous  function $h: [0, 1] \to S^3$ such that $\RQ(h(t))=r(t)$ for all $0 \le t \le 1$ and $h(0)=(1, 0, 0, 0)$. It is easy to see that $h^*(t)=(\cos \pi t, 0, 0, \sin \pi t)$ is that path.

Let $v(t)=h^*(t)\cdot f(r(t))$, then $v(0)=(1, 0, 0, 0)\cdot f(I_{3\times 3})$ and $v(1)=(-1, 0, 0, 0)\cdot f(I_{3\times 3})$, so $v(0)=-v(1)$. $v$ is continuous in $t$. By the intermediate value theorem, there exists $t_0\in [0, 1]$ such that $v(t_0)=0$. So, $d(h^*(t_0), f(r(t_0)))=2\cos^{-1}|v(t_0)|=2\cos^{-1}0=\pi$.

Now we have found $R_0=r(t_0)$ such that the rotations $R_0=\RQ(h^*(t_0))$ and $\RQ(f(R_0))=f(r(t_0)))$ differ by a rotation of $\pi$.
\end{proof}

Note that there exists continuous functions mapping a set of Euler angles to a quaternion representing the same rotation. For example, for extrinsic $x$-$y$-$z$ Euler angles $(\alpha, \beta, \gamma)$, we can get a quaternion of this rotation by multiplying three elemental rotations:
\begin{equation}
\mathbf{Q}_{xyz}(\alpha, \beta, \gamma)=(\cos\frac{\gamma}{2}+\mathbf{k}\sin\frac{\gamma}{2})(\cos\frac{\beta}{2}+\mathbf{j}\sin\frac{\beta}{2})(\cos\frac{\alpha}{2}+\mathbf{i}\sin\frac{\alpha}{2})
\end{equation}

So, as a corollary, we conclude that a continuous function from 3D rotation matrices to Euler angles likewise must produce large error at some point, for otherwise by composing it with $\mathbf{Q}_{xyz}$ we get a continuous function violating theorem \ref{thm:1func}. Let $\mathbf{R}_{xyz}(\alpha, \beta, \gamma):\mathbb{R}^3\to\so(3)$ be the standard conversion from extrinsic $x$-$y$-$z$ Euler angles to rotation matrices given by $\mathbf{R}_{xyz}(\alpha, \beta, \gamma)=\RQ(\mathbf{Q}_{xyz}(\alpha, \beta, \gamma))$.

\begin{cor}
\label{thm:1func_e}
For any continuous function $f: \so(3) \to \mathbb{R}^3$, there exists a rotation $R\in \so(3)$ such that $d(R,\mathbf{R}_{xyz}(f(R)))=\pi$.
\end{cor}

Obviously the same conclusion holds for any possible sequence of Euler angles, intrinsic or extrinsic.

It is easy to see that the same type of argument applies to the simpler case of functions from $\so(2)$ to $\mathbb{R}$ that tries to compute the angle of the rotation, by using the top-left $2\times 2$ of $r(t)$, setting $h^*(t)=2\pi t$ and $v(t)=h^*(t)-r(t)$ and finding $t_0$ such that $v(t_0)=\pi$.

\begin{thm}
For any continuous function $f: \so(2) \to \mathbb{R}$, there exists a rotation $R\in \so(2)$ such that the rotation of angle $f(R)$ differs from $R$
by a rotation of angle $\pi$.
\end{thm}

Given that the quaternion representation of 3D rotations is due to the exceptional isomorphism $\mathrm{Spin}(3)\cong\mathrm{SU}(2)$, there is no obvious generalization to $n$-dimensional rotations. We do however discuss analogous results for 4D rotations in appendix \ref{sec:4d}, due to $\mathrm{Spin}(4)\cong\mathrm{SU}(2)\times\mathrm{SU}(2)$.

\subsection{The Self-selecting Ensemble}

Neural networks are typically continuous (and differentiable) so that gradient based methods can be used for training. Nevertheless, we often employ discontinuous operations at test time, e.g. quantizing a probability distribution into class label in classification networks. Unfortunately this does not work for a regression problem with a continuous output space. However, by employing a simple ensemble trick, it is possible to transform this discontinuity of regression into a discontinuity of classification.

Consider again the problem of recovering the rotation angle from a 2D rotation matrix. Given $M=\left[\begin{smallmatrix}a&-b\\b&a\end{smallmatrix}\right]$ with $a^2+b^2=1$, what we want is exactly $\mathrm{atan2}(b, a)$. There are different ways for choosing the principal value for this multi-valued function, e.g. in $[0, 2\pi)$ or in $(-\pi, \pi]$. Let us distinguish between these two by calling the first one $\atanr$ and the second one $\atanl$. They have discontinuities at rotation angles of $2k\pi$ and $(2k+1)\pi$ ($k\in\mathbb{Z}$), respectively.

Now we can construct two functions that are continuous and computes $\atanr$ and $\atanl$ respectively except when near the discontinuity, where they give incorrect values in order to make them continuous. If we make sure that these two ``wrong regions'' do not overlap, then for any input matrix at least one of the two functions will give a correct rotation angle.

\begin{thm}
\label{thm:2d-2funcs}
There exists continuous functions $f_1, f_2: \so(2) \to \mathbb{R}$ such that for any rotation $R\in \so(2)$, at least one of $f_1(R)$ and $f_2(R)$ gives the correct rotation angle of $R$.
\end{thm}
\begin{proof}
We give an example of $f_1$ and $f_2$ as follows:
\begin{align}
    \allowdisplaybreaks[4]
    f_1\left(\begin{bmatrix}a&-b\\b&a\\\end{bmatrix}\right)&=\begin{cases}
        2\pi-2\cdot\atanr(b, a) & (a<-\frac{1}{2})\\
        \atanl(b, a) & (a\ge-\frac{1}{2})
    \end{cases}\\\nonumber
    f_2\left(\begin{bmatrix}a&-b\\b&a\\\end{bmatrix}\right)&=\begin{cases}
        \atanr(b,a) & (a\le\frac{1}{2})\\
        \pi-2\cdot\atanl(b, a) & (a>\frac{1}{2})
    \end{cases}
\end{align}
It can be checked that these functions are continuous and that their wrong regions do not overlap.
\end{proof}

Since these functions are continuous, they can be approximated by neural networks. On top of these, we can add a classifier that predicts which function would give the correct output for each input. During training time, these functions and the classifier can be trained jointly: the error of the whole ensemble is the sum of the error of each individual functions, weighted by the probability assigned by the classifier. Now the discontinuity only happens at test time, when we select the output of the function with highest assigned probability. We call this method the \emph{self-selecting ensemble}.

Can a similar approach work for the conversion from 3D rotation matrices to quaternions? It turns out that two or even three functions are not enough:

\begin{thm}
\label{thm:3funcs}
For any three continuous functions $f_1, f_2, f_3: \so(3) \to S^3$, there exists a rotation $R\in \so(3)$ such that $d(R,f_i(R))=\pi$ for all $i\in\{1, 2, 3\}$.
\end{thm}
\begin{proof}
Consider functions $v_i: S^3 \to \mathbb{R}$ defined by $v_i(\mathbf{q})=\mathbf{q}\cdot f_i(\RQ(\mathbf{q}))$, for $i=1, 2, 3$. For any
$\mathbf{q} \in S^3$, Since
$\RQ(-\mathbf{q})=\RQ(\mathbf{q})$, $v_i(-\mathbf{q})=-\mathbf{q}\cdot f_i(\RQ(-\mathbf{q}))=-\mathbf{q}\cdot f_i(\RQ(\mathbf{q}))=-v_i(\mathbf{q})$.

Let $V:S^3 \to \mathbb{R}^3$ defined by $V(\mathbf{q})=(v_1(\mathbf{q}), v_2(\mathbf{q}), v_3(\mathbf{q}))$, then $V(-\mathbf{q})=-V(\mathbf{q})$
for any $\mathbf{q} \in S^3$.
By the Borsuk–Ulam theorem,\footnote{See e.g. page 174 of \cite{hatcher2001algebraic}} there exists $\mathbf{q}_0 \in S^3$ such that $V(-\mathbf{q}_0)=V(\mathbf{q}_0)$, then $-V(\mathbf{q}_0)=V(-\mathbf{q}_0)=V(\mathbf{q}_0)$,
so $V(\mathbf{q}_0)=\mathbf{0}$, which means $v_1(\mathbf{q}_0)=v_2(\mathbf{q}_0)=v_3(\mathbf{q}_0)=0$. So for $R_0=\RQ(\mathbf{q}_0)\in \so(3)$, $d(R_0, f_i(R_0))=\pi$ for $i=1, 2, 3$.
\end{proof}

Theorem \ref{thm:3funcs}, with a seemingly simpler proof, implies theorem \ref{thm:1func}. But the proof of Borsuk–Ulam theorem is not simple, and it applies to hyperspheres only while the techniques in theorem \ref{thm:1func} can be useful for other spaces as well, as we will show.

Allowing a fourth function, however, can give us a successful ensemble:

\newcommand{\ThmFourFuncText}{There exists continuous functions $f_1, f_2, f_3, f_4: \so(3) \to S^3$ such that for any rotation $R\in \so(3)$, $\RQ(f_i(R))=R$ for some $i\in\{1, 2, 3, 4\}$.}

\begin{thm}
\label{thm:4funcs}
\ThmFourFuncText
\end{thm}
The proof is by construction. See appendix \ref{sec:proofs}.

Is the same true for Euler angles? Similar to corollary \ref{thm:1func_e} we can conclude that an ensemble of 3 functions will not work. But to find an ensemble of 4 functions that does work, we have to additionally deal with gimbal lock. We first show that there is no correct continuous conversion from rotation matrices to Euler angles when the domain contains a gimbal locked position in its interior.

\newcommand{\ThmGimbalText}{Let $U$ be any neighborhood of $\mathbf{R}_{xyz}(0,\frac{\pi}{2},0)$ in $\so(3)$. There exists no continuous function $f:U\to\mathbb{R}^3$ such that $\mathbf{R}_{xyz}(f(R))=R$ for every $R\in U$.}

\begin{thm}
\label{thm:gimbal}
\ThmGimbalText
\end{thm}

See appendix \ref{sec:proofs} for proof.

The same conclusion can be drawn near any $\mathbf{R}_{xyz}(\alpha,\beta,\gamma)$ where $\beta=\pm\frac{\pi}{2}$. Since gimbal locked positions has to be correctly handled, this problem cannot be solved by simply adding more functions that all have the same gimbal locked positions. Notice that if the order of elemental rotations is $x$-$z$-$y$, then the gimbal locked positions will be different. Analogous to $\mathbf{Q}_{xyz}$ we define $\mathbf{Q}_{xzy}$ as follows:

\begin{equation}
\mathbf{Q}_{xzy}(\alpha, \beta, \gamma)=(\cos\frac{\gamma}{2}+\mathbf{j}\sin\frac{\gamma}{2})(\cos\frac{\beta}{2}+\mathbf{k}\sin\frac{\beta}{2})(\cos\frac{\alpha}{2}+\mathbf{i}\sin\frac{\alpha}{2})
\end{equation}

and also $\mathbf{R}_{xzy}(\alpha, \beta, \gamma)=\RQ(\mathbf{Q}_{xzy}(\alpha, \beta, \gamma))$. We show the existence of a mixed Euler angle ensemble that gives the correct conversion:

\newcommand{\ThmFourFuncEulerText}{There exists continuous functions $f_1, f_2, f_3, f_4: \so(3) \to \mathbb{R}^3$ such that for any rotation $R\in \so(3)$, at least one of the following is equal to $R$: $\mathbf{R}_{xyz}(f_1(R))$, $\mathbf{R}_{xyz}(f_2(R))$, $\mathbf{R}_{xzy}(f_3(R))$ and $\mathbf{R}_{xzy}(f_4(R))$.}

\begin{thm}
\label{thm:4funcs-e}
\ThmFourFuncEulerText
\end{thm}

The proof is by construction. See appendix \ref{sec:proofs}.

In section \ref{sec:exp-1}, we show with experiments that a neural network can successfully learn ensembles of four quaternion representations or ensembles of four mixed-type Euler angle representations that gives small error for rotation matrix conversion over the entire $\so(3)$.

\subsection{Input Symmetry and Effective Input Topology}

The discussion above is of theoretical interest, but would be of little practical relevance if the discontinuity of quaternion and Euler angle representations can be solved by simply using a representation that is continuous, such as the embedding of $\so(3)$ into $\mathbb{R}^5$ proposed in \cite{zhou2019continuity}. We show however, that due to the combined effect of a symmetry in the input and a symmetry of the neural network function, such embeddings can become ineffective.

As in \cite{zhou2019continuity}, we consider the problem of estimating the rotation of a target point cloud relative to a reference point cloud. Since by giving both the target and the reference and considering all possible point cloud at once we will be faced with a extremely complicated input space, for our theoretical analysis we focus on a very simple case: the reference point cloud is fixed, so that the network is only given the target point cloud which is a rotated version of the fixed reference point cloud.

At a first glimpse the topological structure of this problem is exactly the same as converting a 3D rotation matrix into other representations, since there is a homeomorphism between the input space and $\so(3)$. However, the neural network might see a different picture. In a point cloud, there is no assumption of any relationship between different points, and it is considered desirable for the neural network to be invariant under a permutation of input points, which is a design principle of some popular neural network architectures for point cloud processing, e.g. PointNet \cite{qi2017pointnet}.

This behavior causes unexpected consequences. If the input point cloud itself possesses nontrivial rotational symmetry, then different rotations on this point cloud might result in the same set of points, differing only in order. Since the neural network is oblivious to the order of input points, these point clouds generated by different rotations are effectively the same input to the network.

Let $X\subset\mathbb{R}^3$ be a 3D point cloud and $R\in\so(3)$ be a rotation. Let $RX=\{Rx\,|\,x\in X\}$. Then the symmetry group of $X$, $\sym(X)$, is defined by $\sym(X)=\{R\in\so(3)\,|\,RX=X\}$. It is a subgroup of $\so(3)$. Take $X$ as the reference point cloud. Then if two rotations $R_1, R_2\in\so(3)$ generates the same target point cloud, then $R_1X=R_2X$, so $X=R_1^{-1}R_2X$, so $R_1^{-1}R_2\in\sym(X)$, that is, $R_1$ and $R_2$ belong to the same left coset of $\sym(X)$ in $\so(3)$. The reverse is also true.

So for the network, two inputs are equivalent if and only if the rotations that generate them belong to the same left coset of $\sym(X)$. The left cosets of $\sym(X)$ in $\so(3)$ forms a homogeneous space, denoted $\so(3)/\sym(X)$.\footnote{A quotient group is denoted the same way but $\sym(X)$ needs not be a normal subgroup of $\so(3)$ in general so here we do not mean a quotient group.} When $X$ is finite, except for the degenerate case where all points of $X$ lie on a line, $\sym(X)$ must be finite. If $G$ is a finite subgroup of $\so(3)$, then $\so(3)$ is a covering space of $\so(3)/G$, with covering map $p_G:\so(3)\to\so(3)/G,R\mapsto RG$.

We can show that when $\sym(X)$ is nontrivial, a network that is invariant under input point permutation cannot always recover the correct rotation matrix. Here by ``correct'' we mean the rotation given by the network need not be the same as the one used for generating the input point cloud, but must generate the same point cloud up to permutation. $p_G^{-1}[\mathcal{R}]$ denotes the preimage of $\mathcal{R}$ under $p_G$.

\newcommand{\ThmSymText}{Let $G$ be a nontrivial finite subgroup of $\so(3)$, then there does not exist continuous function $f:\so(3)/G\to \so(3)$ such that for every $\mathcal{R}\in\so(3)/G$, $p_G(f(\mathcal{R}))=\mathcal{R}$.}

\begin{thm}
\label{thm:sym}
\ThmSymText
\end{thm}
\begin{proof}
Assume that there exists such a function $f$. Choose any $\mathcal{R}\in\so(3)/G$. Let $R_0=f(\mathcal{R})$. Then $R_0\in p_G^{-1}[\mathcal{R}]$. Since $G$ is a nontrivial group, $\so(3)$ is a nontrivial covering space of $\so(3)/G$. So $|p_G^{-1}[\mathcal{R}]|>1$. Select any $R_1\neq R_0$ from $p_G^{-1}[\mathcal{R}]$. $\so(3)$ is path-connected, so there is a path in $\so(3)$ from $R_0$ to $R_1$, that is, there exists continuous function $r:[0,1]\to\so(3)$ such that $r(0)=R_0$ and $r(1)=R_1$. Fix such an $r$ and let $h=p_G\circ r$. Then $h$ is a path in $\so(3)/G$ and $h(0)=h(1)=\mathcal{R}$.

By the lifting property of covering spaces, $h$ lifts to a unique path in $\so(3)$ starting from $R_0$, which is just $r$. That is, $r$ is the unique continuous function such that $h=p_G\circ r$ and $r(0)=R_0$. We also have $p_G\circ (f\circ h)=(p_G\circ f)\circ h=\mathbf{1}\circ h=h$ and $f(h(0))=f(\mathcal{R})=R_0$. Since $r$ is the unique continuous function such that $h=p_G\circ r$ and $r(0)=R_0$, we must have $r=f\circ h$. But $r(1)=R_1\neq R_0=f(h(1))$, which is a contradiction.
\end{proof}

This is essentially the same proof as in theorem \ref{thm:1func} but without error bounds. Such bounds can be established, but the techniques are much more complicated. We discuss about this in appendix \ref{sec:proofs}.

Similar to corollary \ref{thm:1func_e}, for an embedding $g:\so(3)\to\mathbb{R}^n$, since we can continuously map from the image of $g$ back to $\so(3)$, there exists no continuous function $f$ that finds such an embedding of the correct rotation from the input point cloud, for otherwise $g^{-1}\circ f$ gives a continuous function that computes a correct rotation matrix.

\begin{cor}
Let $G$ be a nontrivial finite subgroup of $\so(3)$ and $g:\so(3)\to\mathbb{R}^n$ be an embedding, then there does not exist continuous function $f:\so(3)/G\to g[\so(3)]$ such that for every $\mathcal{R}\in\so(3)/G$, $p_G(g^{-1}(f(\mathcal{R})))=\mathcal{R}$.
\end{cor}

In particular, This means using the 6D or 5D rotation representations proposed in \cite{zhou2019continuity}, which are embeddings of $\so(3)$ in $\mathbb{R}^6$ or $\mathbb{R}^5$, does not resolve the problem. However, the self-selecting ensemble can solve this problem. We state the following without a formal proof:

\newcommand{\PropSymFourText}{Let $G$ be a finite subgroup of $\so(3)$, then there exists continuous functions $f_1,f_2,f_3,f_4:\so(3)/G\to \so(3)$ such that for every $\mathcal{R}\in\so(3)/G$, $p_G(f_i(\mathcal{R}))=\mathcal{R}$ for some $i\in\{1,2,3,4\}$.}

\begin{pro}
\label{thm:sym-four}
\PropSymFourText
\end{pro}

The idea is that to each $f_i$ is assigned a contractible subset of $\so(3)/G$ on which they give the ``correct'' output. The values on the rest of $\so(3)/G$ are such that $f_i$ is continuous on $\so(3)/G$. Then the $f_i$'s satisfy the requirement if these contractible subsets collectively cover $\so(3)/G$. We discuss more about this in appendix \ref{sec:proofs}.

In section \ref{sec:exp-2}, we test rotation estimation for point clouds using different representations, on one point cloud with trivial symmetry and one with nontrivial rotational symmetry. We show that for the nontrivial case, an ensemble is necessary for the 5D and 6D embeddings as well as quaternions.

\section{Experiments}
\label{sec:exp}

\subsection{Converting Rotation Matrices}
\label{sec:exp-1}

We test the accuracy of converting a 3D rotation matrix into various representations with neural networks, including ensembles. We use an MLP with 5 hidden layers of size 128 each. The size of the output layer varies according to the representation. For ensembles, each individual function as well as the classifier share all their computations except for the output layer, so the overhead of an ensemble over a single network is tiny.

For a single network, the loss function is simply the rotation distance between the input and the output. For an ensemble of $n$ functions, let each individual function in the ensemble be $f_i$ and the classifier be $g$. Here we take the ``raw'' output vector of the classifier, without converting it into a distribution with e.g. $\mathrm{softmax}$. The loss of the whole ensemble on one input rotation $R$ is defined as
\begin{equation}
\label{eqn:loss}
\mathcal{L}(R)=\sum_{i=1}^n \max\{0, g^{(i)}(R)\}\cdot d(R,f_i(R))+\max\{0, 1-\sum_{i=1}^n g^{(i)}(R)\}\cdot \mathcal{L}_{\text{max}}
\end{equation}
where $\mathcal{L}_{\text{max}}=\pi$ is the maximum possible distance between two 3D rotations. At test time, $f_k(R)$ is selected as the output of the ensemble where $k=\mathop{\mathrm{argmax}}_i g^{(i)}(R)$.

We train each network using Adam with learning rate $10^{-4}$ and batch size $256$ for $500,000$ iterations. For training and testing, we sample uniformly from $\so(3)$ (see appendix \ref{sec:rand-rot} for some notes). We sample $100$ million random rotations for testing. 

Figure \ref{fig:t1-all} shows the semi-log plot of errors of each representation by percentile. Mean and maximum errors are given in table \ref{tab:t1-all}. The maximum error is also marked in the graph for clarity. Here we compare a single quaternion, ensemble of four quaternions, a single set of Euler angles, mixed Euler angle ensemble with two sets each of $x$-$y$-$z$ and $x$-$z$-$y$ Euler angles, the 5D embedding and the 6D embedding. Comparison of ensembles of different sizes can be found in appendix \ref{sec:exp-more}. In particular we will show that ensembles of three networks do not work.

The result does not actually show a maximum error of $180^\circ$ for a single quaternion or a single set of Euler angles. This is because while such high errors are guaranteed to exist, they are nevertheless very rare as in general the set of such inputs has measure zero and will almost never be encountered by uniform random sampling. We can see that while using a single quaternion or a single set of Euler angles we inevitably hit the maximum possible error, ensembles of four quaternions or mixed Euler angles can give fairly accurate conversion from rotation matrices to quaternions or Euler angles. In fact, the quaternion ensemble is only marginally worse than the 6D embedding and noticeably better than the 5D embedding.

\subsection{Estimating the Rotation of a Point Cloud}
\label{sec:exp-2}

We test the accuracy of estimating the rotation of a point cloud with various rotation representation, on point clouds with and without symmetries. There are many different possible symmetries. Here we test one of the possibilities as an example. For each of symmetry/non-symmetry we only use one fixed point cloud, for both training and testing, with the rotation being the only variation in the input. To understand why we use such an unconventional setting, along with all the details of network architecture, loss function, training and construction of the point cloud data, please refer to appendix \ref{sec:exp-more}. For now it suffices to know that our network is invariant under input point permutation, and that our experiment comes in two parts: in part one, the point could does not have rotational symmetry; in part two, the point cloud has the rotational symmetry $D_2$ (in Schoenflies notation) which means it is invariant under a rotation of $180^\circ$ around the $x$ axis, the $y$ axis and the $z$ axis.

The result of part 1 is shown in figure \ref{fig:t2-c1} and table \ref{tab:t2-c1}. A single quaternion, ensemble of 4 quaternions, and the 5D and 6D embeddings are compared. The result is consistent with that of rotation matrix conversion.

The result of part 2 is shown in figure \ref{fig:t2-d2} and table \ref{tab:t2-d2}. It can be clearly seen that when the input possessess nontrivial rotational symmetry, the 5D and 6D embeddings can no longer correctly estimate the rotation of the input in all cases.  In contrast, the ensemble of four quaternions continues to perform well. We added the ensemble of four 5D embeddings and the ensemble of four 6D embeddings to the comparison, and the results are similar to the ensemble of four quaternions. We can see that the difference between ensembles and single networks is qualitative while the difference between different kinds of representations is quantitative and comparatively rather minor.

From the numerical result one might guess that with a single network and input with $D_2$ symmetry, the lower bound of maximum error of rotation estimation is $120^\circ$. We will derive this in appendix \ref{sec:proofs}.

\begin{figure}\CenterFloatBoxes
\begin{floatrow}
\ffigbox[1.1\linewidth]{
  \includegraphics[width=\linewidth]{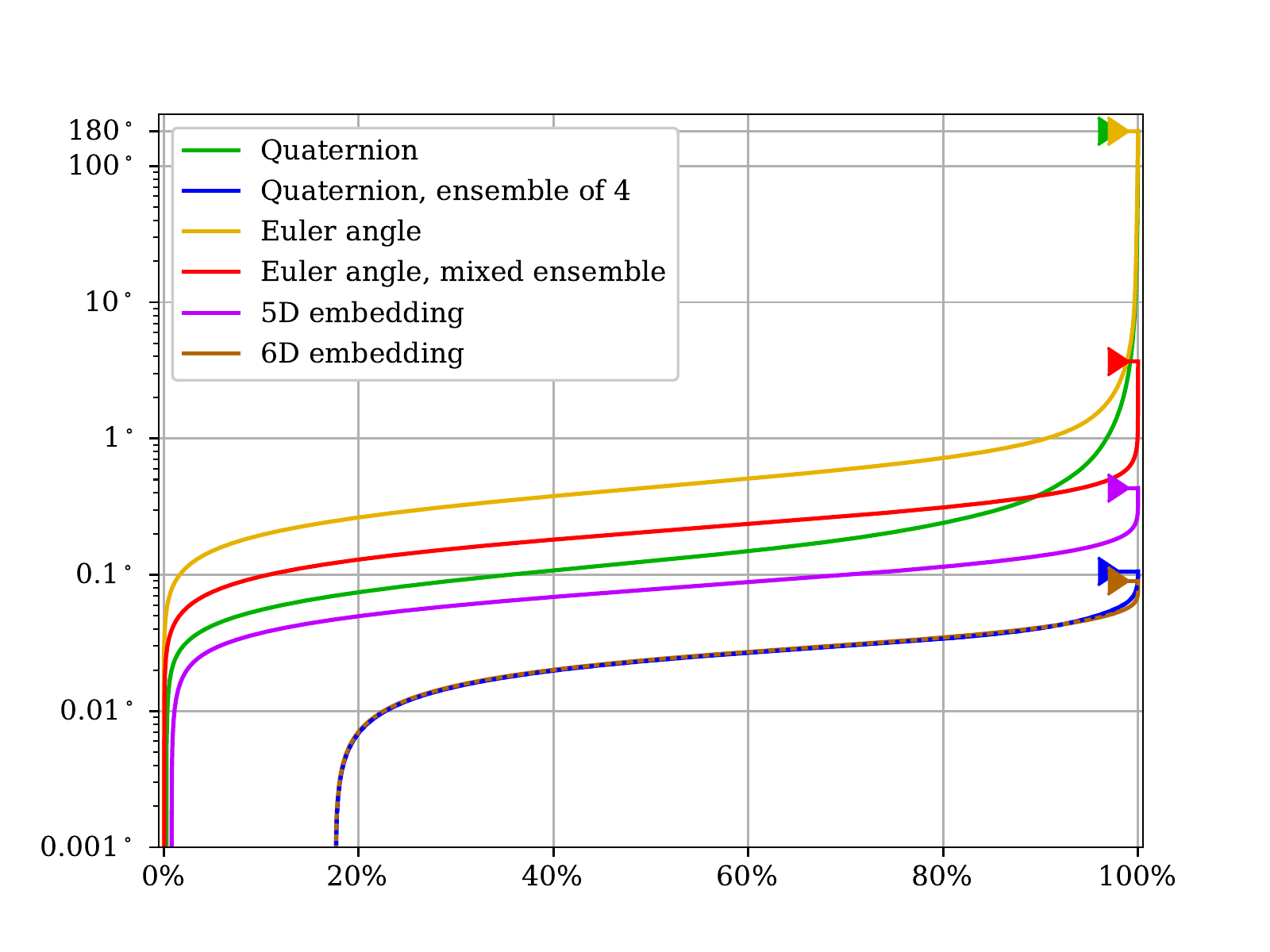}
}
{
\caption{Error of rotation matrix conversion by percentile}
\label{fig:t1-all}
}
\killfloatstyle\ttabbox{
{\footnotesize
\begin{tabular}{lrr}
  \toprule
  \cb{white} Type            & Mean($^\circ$) & Max($^\circ$)  \\
  \midrule
  \cb{clr2} Quat.            & $0.3323$       & $179.9995$ \\
  \cb{clr1} Quat. $\times$4  & $0.0226$       &   $0.1059$ \\
  \cb{clr0} Euler            & $0.7368$       & $179.9981$ \\
  \cb{clr3} Euler mix        & $0.2278$       &   $3.6693$ \\
  \cb{clr4} 5D               & $0.0838$       &   $0.4328$ \\
  \cb{clr5} 6D               & $0.0225$       &   $0.0901$ \\
  \bottomrule
\end{tabular}
}
}
{
\caption{Error statistics}
\label{tab:t1-all}
}
\end{floatrow}
\end{figure}

\begin{figure}\CenterFloatBoxes
\begin{floatrow}
\ffigbox[1.1\linewidth]{
  \includegraphics[width=\linewidth]{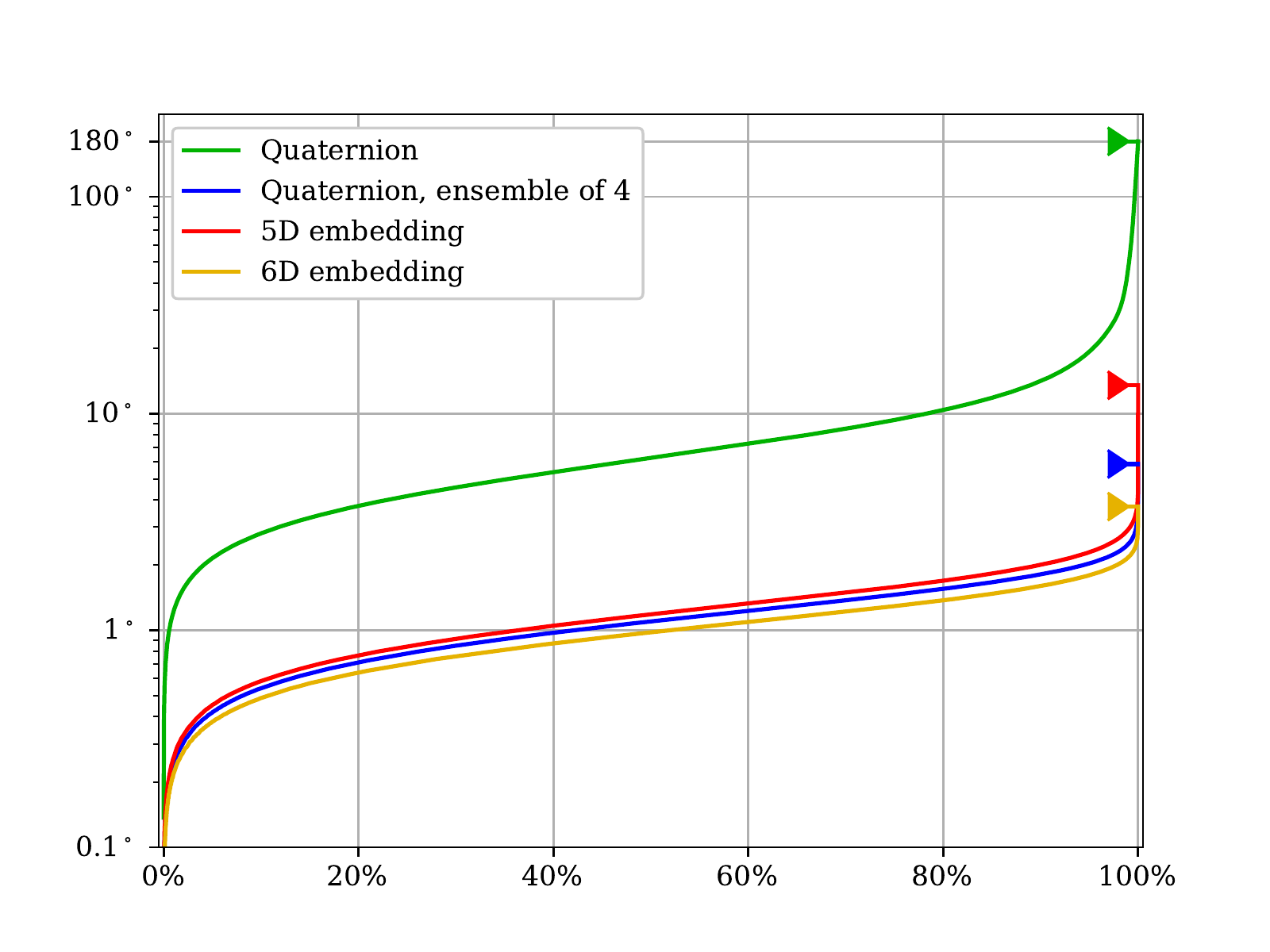}
}
{
\caption{Error of rotation estimation by percentile, of a point cloud with trivial symmetry}
\label{fig:t2-c1}
}
\killfloatstyle\ttabbox{
{\footnotesize
\begin{tabular}{lrr}
  \toprule
  \cb{white} Type            & Mean($^\circ$) & Max($^\circ$)  \\
  \midrule
  \cb{clr2} Quat.            & $8.3963$       & $179.9938$ \\
  \cb{clr1} Quat. $\times$4  & $1.1487$       &   $5.8517$ \\
  \cb{clr3} 5D               & $1.2550$       &  $13.5052$ \\
  \cb{clr0} 6D               & $1.0197$       &   $3.7235$ \\
  \bottomrule
\end{tabular}
}
}
{
\caption{Error statistics}
\label{tab:t2-c1}
}
\end{floatrow}
\end{figure}

\begin{figure}\CenterFloatBoxes
\begin{floatrow}
\ffigbox[1.1\linewidth]{
  \includegraphics[width=\linewidth]{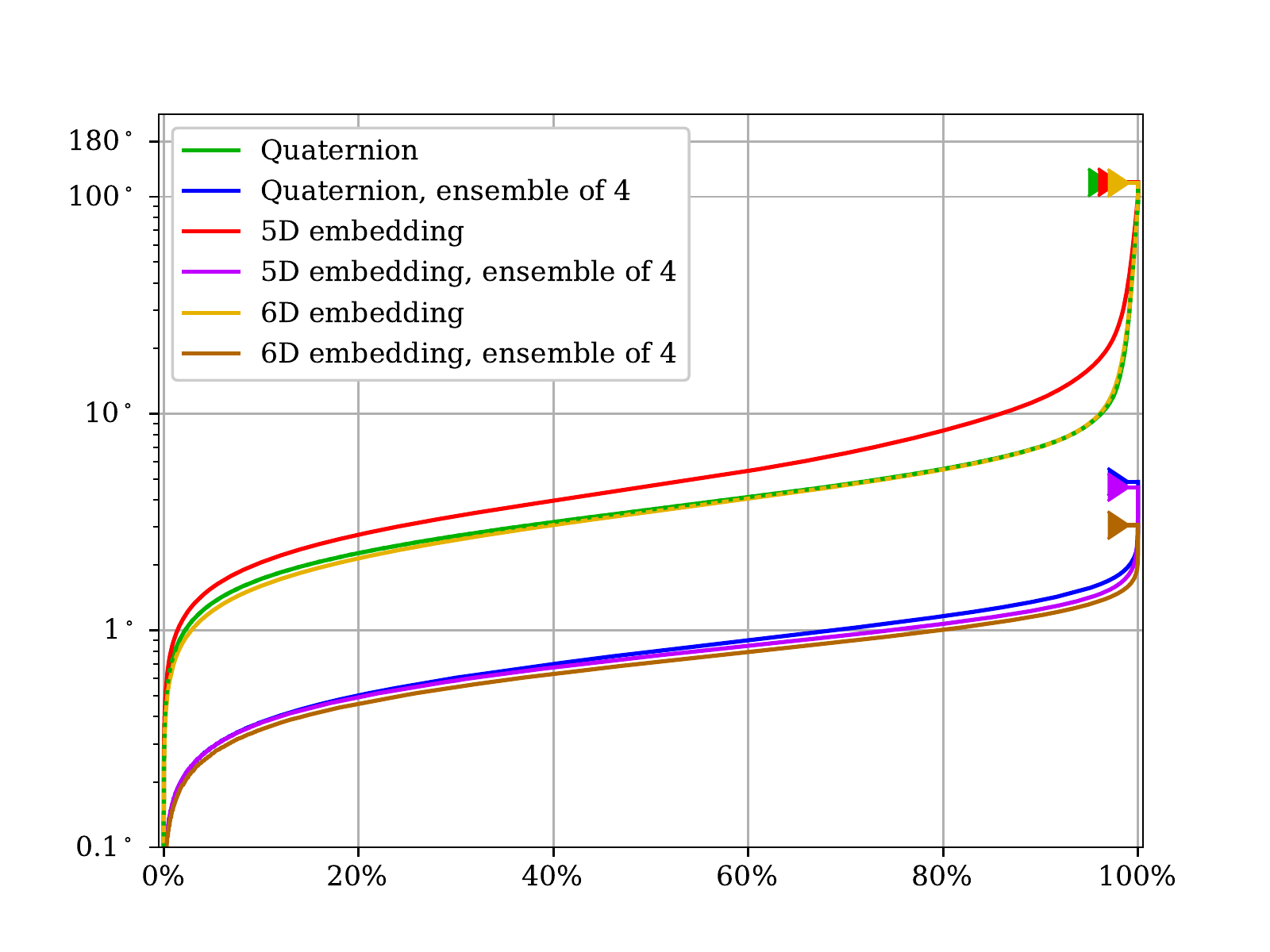}
}
{
\caption{Error of rotation estimation by percentile, of a point cloud with $D_2$ symmetry}
\label{fig:t2-d2}
}
\killfloatstyle\ttabbox{
{\footnotesize
\begin{tabular}{lrr}
  \toprule
  \cb{white} Type            & Mean($^\circ$) & Max($^\circ$)  \\
  \midrule
  \cb{clr2} Quat.            & $4.5983$       & $116.2830$ \\
  \cb{clr1} Quat. $\times$4  & $0.8465$       &   $4.8397$ \\
  \cb{clr3} 5D               & $6.5245$       & $116.5406$ \\
  \cb{clr4} 5D $\times$4     & $0.7958$       &   $4.5727$ \\
  \cb{clr0} 6D               & $4.5320$       & $115.6072$ \\
  \cb{clr5} 6D $\times$4     & $0.7420$       &   $3.0548$ \\
  \bottomrule
\end{tabular}
}
}
{
\caption{Error statistics}
\label{tab:t2-d2}
}
\end{floatrow}
\end{figure}

\section{Conclusion}

In this paper, we analyzed the discontinuity problem of quaternion and Euler angle representation of 3D rotations in neural networks from a topological perspective and showed that a maximum error of $180^\circ$ must occur. We further explored the effect of symmetry in the input on the ability of the network to find the correct rotation representation, and found that symmetries in the input can cause continuous rotation representations to become ineffective, with provable lower bounds of maximum error. We proposed the self-selecting ensemble to solve this discontinuity problem and showed that it works well with different rotation representations, even in the presence of symmetry in the input. We verified our theory with experiments on two simple example problems, the conversion from rotation matrices to other representations and the estimation of the rotation of a point cloud. An extension of our results to 4D rotations were also discussed.

The application of our theoretical analysis and ensemble method to real-world problems involving rotation representation can be a direction for future researches. In addition, the potential usefulness of the self-selecting ensemble for solving a broader range of regression problems with different input and output topology or with other discontinuities in general can be further explored.

\section*{Broader Impact}

This work is mainly concerned with a theoretical problem and does not present any foreseeable societal consequence.


\bibliographystyle{plainnat}
\bibliography{rotation}

\appendix
\newpage

\section{Mathematical Notes}
\label{sec:math}
\subsection{Distance in $\so(3)$}
The ``difference'' between two rotations $R_1$ and $R_2$ can be described by what it takes to change $R_1$ into $R_2$, which can be done by multiplying $R_2R_1^{-1}$ on the left. $R_2R_1^{-1}$ is also a rotation. In 3 dimensions, the rotation angle of $R_2R_1^{-1}$ can be used to measure the distance between $R_1$ and $R_2$.

Rotation matrices are orthogonal matrices. An important property of orthogonal matrices is that all their eigenvalues have length 1. Complex eigenvalues of real matrices always appear in conjugate pairs. For an orthogonal matrix, this means their eigenvalues must be $1$, $-1$ or pairs of complex numbers of the form $\cos\theta\pm i\sin\theta$.

For 3D rotation matrices, their three eigenvalues are exactly $1$ and $\cos\theta\pm i\sin\theta$ where $\theta$ is the rotation angle. The trace of a matrix equals the sum of its eigenvalues. So, for a 3D rotation matrix with rotation angle $\theta$, $\mathrm{tr}(R)=1+(\cos\theta+i\sin\theta)+(\cos\theta-i\sin\theta)=1+2\cos\theta$, so $\theta=\cos^{-1}\frac{\mathrm{tr}(R)-1}{2}$. Let $d(R_1,R_2)$ denote the distance between $R_1$ and $R_2$, then $d(R_1,R_2)=\cos^{-1}\frac{\mathrm{tr}(R_2R_1^{-1})-1}{2}$.

For a unit quaternion $\mathbf{q}=w+\mathbf{i}x+\mathbf{j}y+\mathbf{k}z$, we have
\begin{align}
d(I,\RQ(\mathbf{q}))&=\cos^{-1}\frac{\mathrm{tr}(\RQ(\mathbf{q}))-1}{2} \nonumber \\
&=\cos^{-1}\frac{(1-2y^2-2z^2)+(1-2x^2-2z^2)+(1-2x^2-2y^2)-1}{2} \nonumber \\
&=\cos^{-1}\frac{2-4(x^2+y^2+z^2)}{2} \nonumber \\
&=\cos^{-1}(2w^2-1) \nonumber \\
&=2\cos^{-1}|w|
\end{align}

That is, $d(I,\RQ(\mathbf{q}))=2\cos^{-1}|\mathrm{Re}(\mathbf{q})|$ where $\mathrm{Re}(\mathbf{q})$ is the scalar part of $\mathbf{q}$. We consider distances between rotations represented by quaternions often, so to avoid having to write $\RQ$ every time we let $d$ also take quaternions as arguments directly, in which case by $\mathbf{q}$ we mean $\RQ(\mathbf{q})$. We have
\begin{align}
d(\mathbf{p},\mathbf{q})&=d(\RQ(\mathbf{p}),\RQ(\mathbf{q})) \nonumber \\
&=d(I,\RQ(\mathbf{q})\RQ(\mathbf{p})^{-1}) \nonumber \\
&=d(I,\RQ(\mathbf{q}\overline{\mathbf{p}})) \nonumber \\
&=\cos^{-1}(2\mathrm{Re}(\mathbf{q}\overline{\mathbf{p}})^2-1) \nonumber \\
&=\cos^{-1}(2(\mathbf{p}\cdot\mathbf{q})^2-1) \nonumber \\
&=2\cos^{-1}|\mathbf{p}\cdot\mathbf{q}|
\end{align}

Note however that $d$ defined as such is a metric in $\so(3)$ but not a metric in $S^3$, as it does not satisfy the triangle inequality. Instead, we denote the usual metric on $S^3$, the geodesic distance, by $d_Q$: $d_Q(\mathbf{p},\mathbf{q})=\cos^{-1}(\mathbf{p}\cdot\mathbf{q})$. We have $d(\mathbf{p},\mathbf{q})=\min\{2d_Q(\mathbf{p},\mathbf{q}),2\pi-2d_Q(\mathbf{p},\mathbf{q})\}$, or, $d(\mathbf{p},\mathbf{q})=\min\{2d_Q(\mathbf{p},\mathbf{q}),2d_Q(\mathbf{p},-\mathbf{q})\}$.

\subsection{Distance in $\so(3)/G$}

In theorem \ref{thm:sym} we left open the problem of establishing error bounds. To find such bounds we must first define distance between a rotation and an element of $\so(3)/G$. Let $S\in\so(3)$ and $\mathcal{R}\in\so(3)/G$. Define
\begin{equation}
    d_G(S,\mathcal{R})=\min_{R\in p_G^{-1}[\mathcal{R}]}d(S,R)
\end{equation}
That is, the distance from a rotation $S$ to an element $\mathcal{R}$ of $\so(3)/G$ is its distance to the nearest preimage of $\mathcal{R}$ in $\so(3)$. This can also be extended to have quaternions as the first argument:
\begin{align}
    d_G(\mathbf{s},\mathcal{R})&=d_G(\RQ(\mathbf{s}),\mathcal{R}) \nonumber \\
    &=\min_{R\in p_G^{-1}[\mathcal{R}]}d(\RQ(\mathbf{s}),R) \nonumber \\
    &=\min_{R\in p_G^{-1}[\mathcal{R}]}d(\mathbf{s},R) \nonumber \\
    &=\min_{\mathbf{r}\in \RQ^{-1}[p_G^{-1}[\mathcal{R}]]}d(\mathbf{s},\mathbf{r}) \nonumber \\
    &=\min_{\mathbf{r}\in \RQ^{-1}[p_G^{-1}[\mathcal{R}]]}\min\{2d_Q(\mathbf{s},\mathbf{r}),2d_Q(\mathbf{s},-\mathbf{r})\} \nonumber \\
    &=\min_{\mathbf{r}\in \RQ^{-1}[p_G^{-1}[\mathcal{R}]]}2d_Q(\mathbf{s},\mathbf{r})
\end{align}
On the last line, the inner $\min$ can be absorbed into the outer $\min$ because for any $\mathbf{r}\in\RQ^{-1}[p_G^{-1}[\mathcal{R}]]$ we also have $-\mathbf{r}\in\RQ^{-1}[p_G^{-1}[\mathcal{R}]]$. We then further extend the definition to allow quaternions as the second argument:
\begin{equation}
    d_G(\mathbf{s},\mathbf{r})=d_G(\mathbf{s},p_G(\RQ(\mathbf{r})))=\min_{\mathbf{r}'\in \RQ^{-1}[p_G^{-1}[p_G(\RQ(\mathbf{r}))]]}2d_Q(\mathbf{s},\mathbf{r}')
\end{equation}
Note that $\RQ^{-1}[p_G^{-1}[p_G(\RQ(\mathbf{r}))]]=\{\mathbf{r}\mathbf{q}|\mathbf{q}\in\RQ^{-1}(G)\}$. $\RQ^{-1}(G)$, the preimage of $G$ under the covering map $\RQ:S^3\to\so(3)$, is called a binary polyhedral group. Like specific finite subgroups of $\so(3)$, the specific binary polyhedral groups have their names and notations, but generically let us denote them by $\widehat{G}$. So we have
\begin{equation}
d_G(\mathbf{s},\mathbf{r})=\min_{\mathrm{r}'\in\mathrm{r}\widehat{G}}2d_Q(\mathbf{s},\mathbf{r}')
\end{equation}

\subsection{Distance in $\so(4)$}
To derive analogous results with error bounds for 4D rotations, we need to define the distance between two 4D rotations. In contrast to 3D rotations, there is no single rotation angle and rotation axis in general. Instead, there exists a pair of orthogonal planes that are invariant under the rotation. The rotation matrix has eigenvalues $\cos\theta\pm i\sin\theta$ and $\cos\phi\pm i\sin\phi$, and the restriction of the 4D rotation on each of these two planes is a 2D rotation, one with rotation angle $\theta$ and the other with rotation angle $\phi$.

Note that unlike in 3D where a rotation of angle $-\theta$ is also a rotation of angle $\theta$ around the opposite axis, in 4D if we negate the sign of rotation in one of the pair of invariant planes by flipping the orientation of the other invariant plane, then the sign of rotation on the other invariant plane will also be negated. So in general we cannot necessarily have $\theta$ and $\phi$ both positive without changing the handedness of the coordinate system. Without loss of generality we assume that $|\theta|\ge|\phi|$ and $\theta\in[0,\pi]$.

For two 4D rotations $R_1$ and $R_2$, we compute $\theta$ and $\phi$ from $R_2 R_1^{-1}$ and take $|\theta|+|\phi|$ as the distance between $R_1$ and $R_2$. Let us denote this by $d_4(R_1,R_2)$.

Since we are interested in the quaternion representation of 4D rotations, we want to compute the distance between two rotations from their quaternion representations, without having to compute the eigenvalues of a non-symmetric matrix. We first introduce the quaternion representation of 4D rotations. The formula is due to van Elfrinkhof \cite{van1897eene}. A proof in English can be found in \cite{mebius2005matrix}.

Let $A$ be a $4\times4$ matrix. Define the \emph{associate matrix} $M$ as
\begin{equation}
    \scriptsize
    M=\frac{1}{4}\begin{bmatrix}
        A_{11}\tp A_{22}\tp A_{33}\tp A_{44} & A_{21}\tm A_{12}\tm A_{43}\tp A_{34} & A_{31}\tp A_{42}\tm A_{13}\tm A_{24} & A_{41}\tm A_{32}\tp A_{23}\tm A_{14}\\
        A_{21}\tm A_{12}\tp A_{43}\tm A_{34} & -A_{11}\tm A_{22}\tp A_{33}\tp A_{44} & A_{41}\tm A_{32}\tm A_{23}\tp A_{14} & -A_{31}\tm A_{42}\tm A_{13}\tm A_{24}\\
        A_{31}\tm A_{42}\tm A_{13}\tp A_{24} & -A_{41}\tm A_{32}\tm A_{23}\tm A_{14} & -A_{11}\tp A_{22}\tm A_{33}\tp A_{44} & A_{21}\tp A_{12}\tm A_{43}\tm A_{34}\\
        A_{41}\tp A_{32}\tm A_{23}\tm A_{14} & A_{31}\tm A_{42}\tp A_{13}\tm A_{24} & -A_{21}\tm A_{12}\tm A_{43}\tm A_{34} & -A_{11}\tp A_{22}\tp A_{33}\tm A_{44}\\
    \end{bmatrix}
\end{equation}
$M$ has rank $1$ and Frobenius norm $1$ if and only if $A$ is a rotation matrix. In such case there exist real numbers $a$, $b$, $c$, $d$, $e$, $f$, $g$ and $h$ such that $M=[a\; b\; c\; d]^T[e\; f\; g\; h]$ and $a^2+b^2+c^2+d^2=e^2+f^2+g^2+h^2=1$. The solution is unique up to negating all these numbers. Then, the rotation matrix $A$ can be decomposed as $A=A_L A_R$ where
\begin{equation}
A_L=\begin{bmatrix}
    a & -b & -c & -d \\
    b &  a & -d &  c \\
    c &  d &  a & -b \\
    d & -c &  b &  a \\
\end{bmatrix}\quad
A_R=\begin{bmatrix}
    e & -f & -g & -h \\
    f &  e &  h & -g \\
    g & -h &  e &  f \\
    h &  g & -f &  e \\
\end{bmatrix}
\end{equation}
The two matrices commute, and it can be checked that
\begin{equation}
    A_LA_R[w\; x\mathbf{i}\; y\mathbf{j}\; z\mathbf{k}]^T=(a+b\mathbf{i}+c\mathbf{j}+d\mathbf{k})(w+x\mathbf{i}+y\mathbf{j}+z\mathbf{k})(e+f\mathbf{i}+g\mathbf{j}+h\mathbf{k})
\end{equation}

So, given a pair of two unit quaternions $\mathbf{q}_L=(a+b\mathbf{i}+c\mathbf{j}+d\mathbf{k})$ and $\mathbf{q}_R=(e+f\mathbf{i}+g\mathbf{j}+h\mathbf{k})$, they represent the 4D rotation
\begin{equation}
    \RQQ(\mathbf{q}_L,\mathbf{q}_R)=\begin{bmatrix}
    a & -b & -c & -d \\
    b &  a & -d &  c \\
    c &  d &  a & -b \\
    d & -c &  b &  a \\
\end{bmatrix}
\begin{bmatrix}
    e & -f & -g & -h \\
    f &  e &  h & -g \\
    g & -h &  e &  f \\
    h &  g & -f &  e \\
\end{bmatrix}
\end{equation}
and if $p$ is a point $(w,x,y,z)$ and $\mathbf{p}$ is its quaternion form $w+x\mathbf{i}+y\mathbf{j}+z\mathbf{k}$, then the quaternion form of $\RQQ(\mathbf{q}_L,\mathbf{q}_R)p$ is $\mathbf{q}_L\mathbf{p}\mathbf{q}_R$. $(\mathbf{q}_L,\mathbf{q}_R)$ and $(-\mathbf{q}_L,-\mathbf{q}_R)$ represent the same rotation, and for a given rotation matrix, $(\mathbf{q}_L,\mathbf{q}_R)$ can be uniquely determined up to negation.

We then proceed to find the relationship between quaternion distance $d_Q$ and 4D rotation distance $d_4$. The real Schur decomposition theorem\footnote{See e.g. page 377 of \cite{golub2013matrix}} states that for any real matrix $A\in\mathbb{R}^{n\times n}$, there exists an orthogonal matrix $Q\in\mathbb{R}^{n\times n}$ such that
\begin{equation}
    Q^TAQ=\begin{bmatrix}
        R_{11} & R_{12} & \cdots & R_{1m} \\
        0      & R_{22} & \cdots & R_{2m} \\
        \vdots & \vdots & \ddots & \vdots \\
        0      & 0      & \cdots & R_{mm}
    \end{bmatrix}
\end{equation}
where each $R_{ii}$ is either a $1\times 1$ matrix or a $2\times 2$ matrix having complex conjugate eigenvalues. Apply this to the case where $A$ is a 4D rotation matrix, and treat a pair of eigenvalue $1$ or a pair of eigenvalue $-1$ as a complex conjugate pair, then there exists orthogonal matrix $Q\in\mathbb{R}^{4\times 4}$ such that
\begin{equation}
    Q^TAQ=\begin{bmatrix}
        R_{11} & R_{12} \\
        0      & R_{22}
    \end{bmatrix}
\end{equation}
where $R_{11}$ and $R_{22}$ are $2\times 2$ matrices having complex conjugate eigenvalues. Since $A$ is orthogonal, $Q^TAQ$ is also orthogonal. $\mathrm{det}(Q^TAQ)=\mathrm{det}(A)=1$ so $Q^TAQ$ is a rotation. From $(Q^TAQ)^T(Q^TAQ)=I$ we can get $R_{12}=0$. Assume that $Q$ has determinant $1$ (otherwise it has determinant $-1$, then we can negate the last row of $Q$ so that it has determinant $1$) so that it is also a rotation. $Q^TAQ$ has the same eigenvalues as $A$, so we must have
\begin{equation}
    Q^TAQ=\begin{bmatrix}
        \cos\theta & -\sin\theta & 0 & 0 \\
        \sin\theta & \cos\theta & 0 & 0 \\
        0 & 0 & \cos\phi & -\sin\phi \\
        0 & 0 & \sin\phi & \cos\phi
    \end{bmatrix}
\end{equation}
where $\theta$ and $\phi$ are the two rotation angles of $A$. Let $S=Q^TAQ$. Its associate matrix is
\begin{align}
    M_S&=\frac{1}{4}\begin{bmatrix}
        2\cos\theta+2\cos\phi & 2\sin\theta-2\sin\phi & 0 & 0 \\
        2\sin\theta+2\sin\phi & -2\cos\theta+2\cos\phi & 0 & 0\\
        0 & 0 & 0 & 0\\
        0 & 0 & 0 & 0
    \end{bmatrix}\nonumber\\
    &=\begin{bmatrix}
        \cos\frac{\theta+\phi}{2}\cos\frac{\theta-\phi}{2} & \cos\frac{\theta+\phi}{2}\sin\frac{\theta-\phi}{2} & 0 & 0 \\
        \sin\frac{\theta+\phi}{2}\cos\frac{\theta-\phi}{2} & \sin\frac{\theta+\phi}{2}\sin\frac{\theta-\phi}{2} & 0 & 0 \\
        0 & 0 & 0 & 0 \\
        0 & 0 & 0 & 0
    \end{bmatrix}
\end{align}
From this we can then find the quaternion representation of $S$: 
\begin{equation}
    S=\RQQ(\pm(\mathbf{s}_L,\mathbf{s}_R)),\quad\mathbf{s}_L=\cos\frac{\theta+\phi}{2}+\mathbf{i}\sin\frac{\theta+\phi}{2},\quad\mathbf{s}_R=\cos\frac{\theta-\phi}{2}+\mathbf{i}\sin\frac{\theta-\phi}{2}
\end{equation}
Since $Q$ is also a rotation matrix, it can be represented as a pair of quaternions as well. Assume that $Q=\RQQ(\mathbf{q}_L,\mathbf{q}_R)$.

Now we can find the quaternion representation of $A$. Given a point $p$, the quaternion form of $Ap=QSQ^Tp$ is $\mathbf{q}_L\mathbf{s}_L\overline{\mathbf{q}_L}\mathbf{p}\overline{\mathbf{q}_R}\mathbf{s}_R\mathbf{q}_R$, which tells us that $A=\RQQ(\mathbf{q}_L\mathbf{s}_L\overline{\mathbf{q}_L},\overline{\mathbf{q}_R}\mathbf{s}_R\mathbf{q}_R)$. If one quaternion representation of $A$ is $(\mathbf{u}_L,\mathbf{u}_R)$, then $(\mathbf{u}_L,\mathbf{u}_R)=\pm(\mathbf{q}_L\mathbf{s}_L\overline{\mathbf{q}_L},\overline{\mathbf{q}_R}\mathbf{s}_R\mathbf{q}_R)$.

For any unit quaternions $\mathbf{p}$ and $\mathbf{q}$, $\mathrm{Re}(\mathbf{p}\mathbf{q}\overline{\mathbf{p}})=(\mathbf{p}\mathbf{q})\cdot\mathbf{p}=\mathbf{p}\cdot(\mathbf{p}\mathbf{q})=\mathrm{Re}(\overline{\mathbf{p}}\mathbf{p}\mathbf{q})=\mathrm{Re}(\mathbf{q})$, so either
\begin{gather}
\mathrm{Re}(\mathbf{u}_L)=\mathrm{Re}(\mathbf{q}_L\mathbf{s}_L\overline{\mathbf{q}_L})=\mathrm{Re}(\mathbf{s}_L)=\cos\frac{\theta+\phi}{2}\nonumber\\
\mathrm{Re}(\mathbf{u}_R)=\mathrm{Re}(\overline{\mathbf{q}_R}\mathbf{s}_R\mathbf{q}_R)=\mathrm{Re}(\mathbf{s}_R)=\cos\frac{\theta-\phi}{2}
\end{gather}
or
\begin{gather}
\mathrm{Re}(\mathbf{u}_L)=\mathrm{Re}(-\mathbf{q}_L\mathbf{s}_L\overline{\mathbf{q}_L})=-\mathrm{Re}(\mathbf{s}_L)=-\cos\frac{\theta+\phi}{2}\nonumber\\
\mathrm{Re}(\mathbf{u}_R)=\mathrm{Re}(-\overline{\mathbf{q}_R}\mathbf{s}_R\mathbf{q}_R)=-\mathrm{Re}(\mathbf{s}_R)=-\cos\frac{\theta-\phi}{2}
\end{gather}
In the first case, $d_Q(\mathbf{1},\mathbf{u}_L)=\frac{\theta+\phi}{2}$ and $d_Q(\mathbf{1},\mathbf{u}_R)=\frac{\theta-\phi}{2}$. In the second case, $d_Q(\mathbf{1},\mathbf{u}_L)=\pi-\frac{\theta+\phi}{2}$ and $d_Q(\mathbf{1},\mathbf{u}_R)=\pi-\frac{\theta-\phi}{2}$. We assumed that $\theta\in[0,\pi]$, so we can combine both cases:
\begin{align}
    |\theta|&=\min\{d_Q(\mathbf{1},\mathbf{u}_L)+d_Q(\mathbf{1},\mathbf{u}_R),2\pi-d_Q(\mathbf{1},\mathbf{u}_L)-d_Q(\mathbf{1},\mathbf{u}_R)\}\nonumber\\
    |\phi|&=|d_Q(\mathbf{1},\mathbf{u}_L)-d_Q(\mathbf{1},\mathbf{u}_R)|
\end{align}

So, in terms of quaternions, $d_4$ is defined as
\begin{align}
    &d_4((\mathbf{p}_L,\mathbf{p}_R),(\mathbf{q}_L,\mathbf{q}_R))\nonumber\\
    =&\min\{d_Q(\mathbf{p}_L,\mathbf{q}_L)+d_Q(\mathbf{p}_R,\mathbf{q}_R),2\pi-d_Q(\mathbf{p}_L,\mathbf{q}_L)-d_Q(\mathbf{p}_R,\mathbf{q}_R)\}\nonumber\\
    &+|d_Q(\mathbf{p}_L,\mathbf{q}_L)-d_Q(\mathbf{p}_R,\mathbf{q}_R)|
\end{align}

\subsection{Sampling Random Rotations}
\label{sec:rand-rot}
We would like to ensure that random 3D rotations are sampled properly for training and testing. \cite{randomrotation} gives a discussion of what it means for a distribution of 3D rotations to be ``uniform''. We note here that uniformly sampling $\alpha\in(-\pi,\pi]$, $\beta\in[-\frac{\pi}{2},\frac{\pi}{2}]$ and $\gamma\in(-\pi,\pi]$ does not result in $\mathbf{R}_{xyz}(\alpha,\beta,\gamma)$ being uniformly distributed in $\so(3)$.

Uniformly sampling a rotation axis from $S^2$ and an angle from $[0,\pi]$ does not give a uniformly distributed random rotation in $\so(3)$ either. This seemingly correct method actually heavily favors rotations with small angles.

There is in fact a easy and correct way to sample a uniform random rotation in $\so(3)$. The covering map $\RQ:S^3\to\so(3)$ is a local isometry, so if we uniformly sample a unit quaternion $\mathbf{q}$ from $S^3$, then $\RQ(\mathbf{q})$ is a uniform random rotation in $\so(3)$. To uniformly sample a unit quaternion, sample a 4D vector from a standard normal distribution and normalize it to have unit length.

\section{Additional Theoretical Results}
\label{sec:proofs}
In this section we give proofs omitted in the main text, as well as some additional remarks.
\begin{customthm}{\ref{thm:4funcs}}
\ThmFourFuncText
\end{customthm}
\begin{proof}
We give an example of such a set of functions. For clarity, we define them as functions of quaternions. If we ensure that $f_i(\mathbf{q})=f_i(\mathbf{-q})$ for all $\mathbf{q}\in S^3$ and $i\in\{1, 2, 3, 4\}$, they will be well-defined functions of rotations. $N$ is the normalization function $N(\mathbf{q})=\frac{\mathbf{q}}{||\mathbf{q}||}$ defined for all $\mathbf{q}\neq\mathbf{0}$. Let
\begin{gather}
    \allowdisplaybreaks
    f_1(a, b, c, d)=
    \begin{cases}
         (   a,   b,   c,   d) & (a \ge \frac{1}{2}) \\
        N( 1-a, 2ab, 2ac, 2ad) & (0 \le a < \frac{1}{2})\\
        N( 1+a, 2ab, 2ac, 2ad) & (-\frac{1}{2} \le a < 0)\\
         (  -a,  -b,  -c,  -d) & (a < -\frac{1}{2}) \\
    \end{cases}\\\nonumber
    f_2(a, b, c, d)=
    \begin{cases}
         (   a,   b,   c,   d) & (b \ge \frac{1}{2}) \\
        N( 2ab, 1-b, 2bc, 2bd) & (0 \le b < \frac{1}{2})\\
        N( 2ab, 1+b, 2bc, 2bd) & (-\frac{1}{2} \le b < 0)\\
         (  -a,  -b,  -c,  -d) & (b < -\frac{1}{2}) \\
    \end{cases}\\\nonumber
    f_3(a, b, c, d)=
    \begin{cases}
         (   a,   b,   c,   d) & (c \ge \frac{1}{2}) \\
        N( 2ac, 2bc, 1-c, 2cd) & (0 \le c < \frac{1}{2})\\
        N( 2ac, 2bc, 1+c, 2cd) & (-\frac{1}{2} \le c < 0)\\
         (  -a,  -b,  -c,  -d) & (c < -\frac{1}{2}) \\
    \end{cases}\\\nonumber
    f_4(a, b, c, d)=
    \begin{cases}
         (   a,   b,   c,   d) & (d \ge \frac{1}{2}) \\
        N( 2ad, 2bd, 2cd, 1-d) & (0 \le d < \frac{1}{2})\\
        N( 2ad, 2bd, 2cd, 1+d) & (-\frac{1}{2} \le d < 0)\\
         (  -a,  -b,  -c,  -d) & (d < -\frac{1}{2}) \\
    \end{cases}
\end{gather}

Check that these functions are indeed continuous at case boundaries and that $f_i(-\mathbf{q})=f_i(\mathbf{q})$. For a unit quaternion $(a, b, c, d)$, $a^2+b^2+c^2+d^2=1$, so $\max\{a^2,b^2,c^2,d^2\}\ge\frac{1}{4}$. If $a^2$, $b^2$, $c^2$ or $d^2$ is at least $\frac{1}{4}$, then $f_1$, $f_2$, $f_3$ or $f_4$ respectively will give the correct output.
\end{proof}

\begin{customthm}{\ref{thm:gimbal}}
\ThmGimbalText
\end{customthm}
\begin{proof}
Assume that such a function $f$ exists. It can be checked that $\mathbf{Q}_{xyz}(\theta,\phi,\theta)\cdot\mathbf{Q}_{xyz}(0,\frac{\pi}{2},0)=\cos(\frac{\pi}{4}-\frac{\phi}{2})$, so for $0\le\phi\le\pi$, $d(\mathbf{R}_{xyz}(\theta,\phi,\theta),\mathbf{R}_{xyz}(0,\frac{\pi}{2},0))=\frac{\pi}{2}-\phi$. So there exists $\epsilon>0$ such that $\mathbf{R}_{xyz}(\theta,\phi,\theta)\in U$ for all $\theta\in\mathbb{R}$ and $\frac{\pi}{2}-\epsilon<\phi<\frac{\pi}{2}$.

If $\mathbf{R}_{xyz}(\alpha,\beta,\gamma)=\mathbf{R}_{xyz}(\alpha',\beta',\gamma')$, but $\beta\neq k\pi+\frac{\pi}{2}$ for any $k\in\mathbb{Z}$, we must have $(\alpha',\beta',\gamma')=(\alpha+2m\pi,\beta+2n\pi,\gamma+2k\pi)$ or $(\alpha',\beta',\gamma')=(\alpha+(2m+1)\pi,-\beta+(2n+1)\pi,\gamma+(2k+1)\pi)$ for some $m,n,k\in\mathbb{Z}$.
In any case, let 
\begin{equation}
g(\theta,\phi)=\frac{f^{(1)}(\mathbf{R}_{xyz}(\theta,\phi,\theta))-\theta}{\pi}
\end{equation}
where $f^{(1)}$ is the first coordinate of $f$. Then for $\theta\in\mathbb{R}$ and $\frac{\pi}{2}-\epsilon<\phi<\frac{\pi}{2}$, $g$ is continuous and integer-valued, so it must be constant. But since $\mathbf{R}_{xyz}(0,\phi,0)=\mathbf{R}_{xyz}(2\pi,\phi,2\pi)$, we have $g(0,\phi)-g(2\pi,\phi)=2$, which is a contradiction.
\end{proof}

\begin{customthm}{\ref{thm:4funcs-e}}
\ThmFourFuncEulerText
\end{customthm}
\begin{proof}
Similar to theorem \ref{thm:2d-2funcs} we distinguish between two versions of $\atan$ with different choices of principal values. Let $\atanl$ be the version with principal value in $(-\pi,\pi]$ and $\atanr$ be the version with principal value in $[0, 2\pi)$.

Let $\mathcal{T}$ be a parametrized function, defined as follows:
\begin{equation}
    \mathcal{T}(a_1,a_2,a_3,a_4)(x)=\begin{cases}
        0 & (x \le a_1) \\\vspace{2pt}
        \frac{x-a_1}{a_2-a_1} & (a_1 < x \le a_2) \\\vspace{2pt}
        1 & (a_2 < x \le a_3) \\\vspace{2pt}
        \frac{a_4-x}{a_4-a_3} & (a_3 < x \le a_4) \\\vspace{2pt}
        0 & (a_4 < x) \\
    \end{cases}
\end{equation}
That is, $\mathcal{T}(a_1,a_2,a_3,a_4)$ is a piecewise linear function defined by connecting the points $(a_1, 0)$, $(a_2, 1)$, $(a_3, 1)$, $(a_4, 0)$ in $[a_1, a_4]$ and constantly $0$ beyond that range. Define these instances of $\mathcal{T}$:
\begin{align}
\allowdisplaybreaks
\mathcal{T}_{\alpha\triangleleft}&=\mathcal{T}(-\frac{2\pi}{3},-\frac{\pi}{2},\frac{\pi}{2},\frac{2\pi}{3})\\
\mathcal{T}_{\alpha\triangleright}&=\mathcal{T}(\frac{\pi}{3},\frac{\pi}{2},\frac{3\pi}{2},\frac{5\pi}{3})\nonumber\\
\mathcal{T}_{\beta}&=\mathcal{T}(-\frac{\pi}{3},-\frac{\pi}{4},\frac{\pi}{4},\frac{\pi}{3})\nonumber\\
\mathcal{T}_{\gamma\triangleleft}&=\mathcal{T}(-\frac{5\pi}{6},-\frac{3\pi}{4},\frac{3\pi}{4},\frac{5\pi}{6})\nonumber\\
\mathcal{T}_{\gamma\triangleright}&=\mathcal{T}(\frac{\pi}{6},\frac{\pi}{4},\frac{7\pi}{4},\frac{11\pi}{6})\nonumber
\end{align}

Let $M\in\so(3)$ be a rotation matrix. Define these functions:
\begin{align}
\alpha_1(M)&=\begin{cases}
    \atanl(M_{32},M_{33}) & (M_{31} \neq \pm1)\\
    \pi & (M_{31} = \pm1)\\
\end{cases}\\
\alpha_2(M)&=\begin{cases}
    \atanr(M_{32},M_{33}) & (M_{31} \neq \pm1)\\
    0 & (M_{31} = \pm1)\\
\end{cases}\nonumber\\
\beta_1(M)&=-\sin^{-1}M_{31}\nonumber\\
\gamma_1(M)&=\begin{cases}
    \atanl(M_{21},M_{11}) & (M_{31} \neq \pm1)\\
    \pi & (M_{31} = \pm1)\\
\end{cases}\nonumber\\
t_1(M)&=\mathcal{T}_{\alpha\triangleleft}(\alpha_1(M))\cdot\mathcal{T}_{\beta}(\beta_1(M))\cdot\mathcal{T}_{\gamma\triangleleft}(\gamma_1(M))\nonumber\\
t_2(M)&=\mathcal{T}_{\alpha\triangleright}(\alpha_2(M))\cdot\mathcal{T}_{\beta}(\beta_1(M))\cdot\mathcal{T}_{\gamma\triangleleft}(\gamma_1(M))\nonumber\\
f_1(M)&=(\alpha_1(M),\beta_1(M),\gamma_1(M))\cdot t_1(M)\nonumber\\
f_2(M)&=(\alpha_2(M),\beta_1(M),\gamma_1(M))\cdot t_2(M)\nonumber\displaybreak[0]\\
\alpha_3(M)&=\begin{cases}
    \atanl(-M_{23},M_{22}) & (M_{21} \neq \pm1)\\
    \pi & (M_{21} = \pm1)\\
\end{cases}\nonumber\\
\alpha_4(M)&=\begin{cases}
    \atanr(-M_{23},M_{22}) & (M_{21} \neq \pm1)\\
    0 & (M_{21} = \pm1)\\
\end{cases}\nonumber\\
\beta_3(M)&=\sin^{-1}M_{21}\nonumber\\
\gamma_3(M)&=\begin{cases}
    \atanr(M_{31},M_{11}) & (M_{21} \neq \pm1)\\
    0 & (M_{21} = \pm1)\\
\end{cases}\nonumber\\
t_3(M)&=\mathcal{T}_{\alpha\triangleleft}(\alpha_3(M))\cdot\mathcal{T}_{\beta}(\beta_3(M))\cdot\mathcal{T}_{\gamma\triangleright}(\gamma_3(M))\nonumber\\
t_4(M)&=\mathcal{T}_{\alpha\triangleright}(\alpha_4(M))\cdot\mathcal{T}_{\beta}(\beta_3(M))\cdot\mathcal{T}_{\gamma\triangleright}(\gamma_3(M))\nonumber\\
f_3(M)&=(\alpha_3(M),\beta_3(M),\gamma_3(M))\cdot t_3(M)\nonumber\\
f_4(M)&=(\alpha_4(M),\beta_3(M),\gamma_3(M))\cdot t_4(M)\nonumber
\end{align}
We prove that $f_1$, $f_2$, $f_3$ and $f_4$ defined as such meet the requirements in the theorem. Consider the continuity of $f_1$. We divide $\so(3)$ into overlapping regions according to the range of $\alpha$, $\beta$ and $\gamma$ in their Euler angle representation. Let
\begin{equation}
M=\mathbf{R}_{xyz}(\alpha,\beta,\gamma)=\begin{bmatrix}
    c_\beta c_\gamma & -c_\alpha s_\gamma + s_\alpha s_\beta c_\gamma & s_\alpha s_\gamma + c_\alpha s_\beta c_\gamma \\
    c_\beta s_\gamma & c_\alpha c_\gamma + s_\alpha s_\beta s_\gamma  & -s_\alpha c_\gamma + c_\alpha s_\beta s_\gamma \\
    -s_\beta         & s_\alpha c_\beta                               & c_\alpha c_\beta
\end{bmatrix}
\end{equation}
Where $s_\theta=\sin\theta$ and $c_\theta=\cos\theta$.

\emph{Case 1: $\alpha\in(-\pi,\pi)$, $\beta\in(-\frac{\pi}{2},\frac{\pi}{2})$, $\gamma\in(-\pi,\pi)$}.
$M_{31}\neq\pm 1$ and the discontinuities of $\atanl(M_{32},M_{33})$ and $\atanl(M_{21},M_{11})$ are avoided, so $\alpha_1$, $\beta_1$ and $\gamma_1$ are continuous. $\mathcal{T}$ is always continuous, so $t_1$ and thus $f_1$ are continuous.

\emph{Case 2: $\alpha\in(-\frac{4\pi}{3},-\frac{2\pi}{3})\cup(\frac{2\pi}{3},\frac{4\pi}{3})$}. $\mathcal{T}_{\alpha\triangleleft}(\alpha_1(M))=0$, so $f_1(M)=(0, 0, 0)$ is constant and continuous.

\emph{Case 3: $\beta\in(-\frac{2\pi}{3},-\frac{\pi}{3})\cup(\frac{\pi}{3},\frac{2\pi}{3})$}. $\mathcal{T}_{\beta}(\beta_1(M))=0$, so $f_1(M)=(0, 0, 0)$ is constant and continuous.

\emph{Case 4: $\gamma\in(-\frac{7\pi}{6},-\frac{5\pi}{6})\cup(\frac{5\pi}{6},\frac{7\pi}{6})$}. $\mathcal{T}_{\gamma\triangleleft}(\gamma_1(M))=0$, so $f_1(M)=(0, 0, 0)$ is constant and continuous.

These four cases collectively cover the entire $\so(3)$, so $f_1$ is continuous in $\so(3)$. Additionally, when \rule{0pt}{10pt}  $\alpha\in[-\frac{\pi}{2},\frac{\pi}{2}]$, $\beta\in[-\frac{\pi}{4},\frac{\pi}{4}]$ and $\gamma\in[-\frac{3\pi}{4},\frac{3\pi}{4}]$, $\alpha_1(M)=\alpha$, $\beta_1(M)=\beta$, $\gamma_1(M)=\gamma$ and $\mathcal{T}_{\alpha\triangleleft}(\alpha_1(M))=\mathcal{T}_{\beta}(\beta_1(M))=\mathcal{T}_{\gamma\triangleleft}(\gamma_1(M))=1$, \rule{0pt}{10pt} so $f_1(M)=(\alpha,\beta,\gamma)$.

Similarly, we can prove that $f_2$, $f_3$ and $f_4$ are continuous, and each have a specific region in $\so(3)$ where they give the correct Euler angles. Several things in the proof above need to be changed for each function accordingly:

For $f_2$ and $f_4$, the range of $\alpha$ in case 1 is $(0,2\pi)$, the range of $\alpha$ in case 4 is $(-\frac{\pi}{3},\frac{\pi}{3})\cup(\frac{5\pi}{3},\frac{7\pi}{3})$ and the range of $\alpha$ in the correct range is $[\frac{\pi}{2},\frac{3\pi}{2}]$.

For $f_3$ and $f_4$, divide $\so(3)$ by angle range of $\mathbf{R}_{xzy}$ instead of $\mathbf{R}_{xyz}$:
\begin{equation}
M=\mathbf{R}_{xzy}(\alpha,\beta,\gamma)=\begin{bmatrix}
    c_\beta c_\gamma & -c_\alpha s_\beta c_\gamma - s_\alpha s_\gamma & s_\alpha s_\beta c_\gamma - c_\alpha s_\gamma \\
    s_\beta          & c_\alpha c_\beta                               & -s_\alpha c_\beta \\
    c_\beta s_\gamma & -c_\alpha s_\beta s_\gamma + s_\alpha c_\gamma & s_\alpha s_\beta s_\gamma + c_\alpha c_\gamma
\end{bmatrix}
\end{equation}
the range of $\gamma$ in case 1 is $(0,2\pi)$, the range of $\gamma$ in case 4 is $(-\frac{\pi}{6},\frac{\pi}{6})\cup(\frac{11\pi}{6},\frac{13\pi}{6})$, and the range of $\gamma$ in the correct range is $[\frac{\pi}{4},\frac{7\pi}{4}]$.

Now we need to prove that any rotation in $\so(3)$ falls within the correct range of at least one of $f_1$, $f_2$, $f_3$ and $f_4$. For $M\in\so(3)$, let $m=\max\{M_{11},-M_{11},M_{21},-M_{21},M_{31},-M_{31}\}$.

If $m=M_{11}$, then $2M_{31}^2 \le M_{11}^2+M_{31}^2 \le M_{11}^2+M_{21}^2+M_{31}^2=1$, so $-\frac{\sqrt{2}}{2}\le M_{31}\le\frac{\sqrt{2}}{2}$, so $-\frac{\pi}{4}\le\beta_1(M)\le\frac{\pi}{4}$. $-M_{11}\le M_{21}\le M_{11}$, so $-\frac{\pi}{4}\le\gamma_1(M)\le\frac{\pi}{4}$.

If additionally $M_{33}\ge 0$, then $-\frac{\pi}{2}\le\alpha_1(M)\le\frac{\pi}{2}$ so $M$ is in the correct range of $f_1$. Otherwise $\frac{\pi}{2}\le\alpha_2(M)\le\frac{3\pi}{2}$ so $M$ is in the correct range of $f_2$.

Other cases can be analyzed similarly. We summarize the result in table \ref{tab:thm-4func-e}. If $M$ satisfies the conditions for multiple functions, then it falls in the correct range of each of them.

In conclusion, $f_1$, $f_2$, $f_3$ and $f_4$ are continuous functions from $\so(3)$ to $\mathbb{R}^3$ and for any rotation $R\in \so(3)$, at least one of $\mathbf{R}_{xyz}(f_1(R))$, $\mathbf{R}_{xyz}(f_2(R))$, $\mathbf{R}_{xzy}(f_3(R))$ and $\mathbf{R}_{xzy}(f_4(R))$ is equal to $R$.
\end{proof}

\renewcommand{\arraystretch}{1.5}
\begin{table}
\begin{floatrow}
\ttabbox{
    \footnotesize
    \centering
    \begin{tabular}{|c|c|c|c|}\hline
    \multicolumn{4}{|c|}{Let $m=\max\{M_{11},-M_{11},M_{21},-M_{21},M_{31},-M_{31}\}$} \\\hline
    \multicolumn{2}{|c|}{$m=-M_{21},M_{11},M_{21}$} & \multicolumn{2}{c|}{$m=-M_{31},-M_{11},M_{31}$} \\\hline
    \multicolumn{2}{|c|}{$\beta_1(M)\in[-\frac{\pi}{4},\frac{\pi}{4}]$} & \multicolumn{2}{c|}{$\beta_3(M)\in[-\frac{\pi}{4},\frac{\pi}{4}]$} \\
    \multicolumn{2}{|c|}{$\gamma_1(M)\in[-\frac{3\pi}{4},-\frac{\pi}{4}],[-\frac{\pi}{4},\frac{\pi}{4}],[\frac{\pi}{4},\frac{3\pi}{4}]$ resp.} & \multicolumn{2}{c|}{$\gamma_3(M)\in[\frac{\pi}{4},\frac{3\pi}{4}],[\frac{3\pi}{4},\frac{5\pi}{4}],[\frac{5\pi}{4},\frac{7\pi}{4}]$ resp.} \\\hline
    $M_{33}\ge 0$ & $M_{33}\le 0$ & $M_{22}\ge 0$ & $M_{22}\le 0$ \\\hline
    $\alpha_1(M)\in[-\frac{\pi}{2},\frac{\pi}{2}]$ & $\alpha_2(M)\in[\frac{\pi}{2},\frac{3\pi}{2}]$ & $\alpha_3(M)\in[-\frac{\pi}{2},\frac{\pi}{2}]$ & $\alpha_4(M)\in[\frac{\pi}{2},\frac{3\pi}{2}]$ \\
    $\mathbf{R}_{xyz}(f_1(M))=M$ & $\mathbf{R}_{xyz}(f_2(M))=M$ & $\mathbf{R}_{xzy}(f_3(M))=M$ & $\mathbf{R}_{xzy}(f_4(M))=M$ \\\hline
    \end{tabular}
}
{
\caption{Deciding which function would be correct for $M$}
\label{tab:thm-4func-e}
}
\end{floatrow}
\end{table}
\renewcommand{\arraystretch}{1.0}

\begin{customthm}{\ref{thm:sym}}
\ThmSymText
\end{customthm}
\textit{Remark.} In addition to showing that these functions do not exist, we also want to know if for a continuous function there is a provable maximum error that must be achieved on some input. Such error bounds will be specific to each possible choice of $G$, so we first list these possibilities. There are two infinite series of finite subgroups of $\so(3)$, plus three isolated ones. We list them, with their Schoenflies notations:
\begin{itemize}
    \item $C_n$, the cyclic groups, the rotational symmetry group of a right pyramid with a regular $n$-sided base.
    \item $D_n$, the dihedral groups, the rotational symmetry group of a right prism with a regular $n$-sided base.
    \item $T$, the chiral tetrahedral group, the rotational symmetry group of a regular tetrahedron.
    \item $O$, the chiral octahedral group, the rotational symmetry group of a cube or a regular octahedron.
    \item $I$, the chiral icosahedral group, the rotational symmetry group of a regular dodecahedron or icosahedron.
\end{itemize}
These are all different except for $C_2=D_1$. Our derivation is valid for the cases where the group contains an element of order $2$ (a $180^\circ$ rotation), that is, all cases except for $C_n$ where $n$ is odd.

Assume that $G$ is one of the applicable groups. Let $f$ be any continuous function from $\so(3)/G$ to $\so(3)$. Now pick one of the two preimages of $f(p_G(I_{3\times 3}))$ under $\RQ$ as $\mathbf{s}$. Define $\widehat{f}:S^3\to S^3$ by the following method:

For any $\mathbf{r}\in S^3$, let $h$ be a path from $\mathbf{1}$ to $\mathbf{r}$, that is, $h$ is a continuous function from $[0,1]$ to $S^3$ such that $h(0)=\mathbf{1}$ and $h(1)=\mathbf{r}$. Then $f \circ p_G \circ \RQ \circ h$ is a path in $\so(3)$ starting from $f(p_G(I_{3\times 3}))=\RQ(\mathbf{s})$. By the lifting property of covering spaces, $f \circ p_G \circ \RQ \circ h$ lifts to a unique path in $S^3$ starting from $\mathbf{s}$. Let $g$ be that path. We prove that the value of $g(1)$ is independent of the choice of $h$:

for any two paths $h_1$ and $h_2$ from $\mathbf{1}$ to $\mathbf{r}$, suppose that $f \circ p_G \circ \RQ \circ h_1$ lifts to $g_1$ from $\mathbf{s}$ and $f \circ p_G \circ \RQ \circ h_2$ lifts to $g_2$ from $\mathbf{s}$. Since $S^3$ is simply connected, there exists a homotopy of paths from $h_1$ to $h_2$. That is, there exists a continuous function $H:[0,1]\times[0,1]\to S^3$ such that $H(t,0)=h_1(t)$ and $H(t,1)=h_2(t)$ for all $0\le t\le 1$ and $H(0,u)=\mathbf{1}$ and $H(1,u)=\mathbf{r}$ for all $0\le u\le 1$. Now let $N=f \circ p_G \circ \RQ \circ H$, then $N:[0,1]\times[0,1]\to \so(3)$ is a homotopy of paths in $\so(3)$,
$N(\cdot,0)=f \circ p_G \circ \RQ \circ h_1$ and $N(\cdot,1)=f \circ p_G \circ \RQ \circ h_2$.

$\RQ\circ g_1=N(\cdot,0)$. By the lifting property of covering spaces, $N$ lifts to a unique homotopy $\widehat{N}$ in $S^3$ such that $\widehat{N}(\cdot,0)=g_1$. Since $\RQ(\widehat{N}(0,u))=\RQ(\mathbf{s})$, for any $u$ we must have $\widehat{N}(0,u)=\mathbf{s}$ or $\widehat{N}(0,u)=-\mathbf{s}$. $\widehat{N}(0,u)$ is also continuous in $u$, so it must be constant. So we must have $\widehat{N}(0,1)=\widehat{N}(0,0)=\mathbf{s}$. Likewise we must have $\widehat{N}(1,1)=\widehat{N}(1,0)$.

$\widehat{N}(\cdot,1)$ is a lift of $N(\cdot,1)$ starting from $\widehat{N}(0,1)=\mathbf{s}$. Since $g_2$ is the unique such lift, we have $\widehat{N}(\cdot,1)=g_2$. So, $g_2(1)=\widehat{N}(1,1)=\widehat{N}(1,0)=g_1(1)$.

Now we have proven that the value of $g(1)$ is independent of the choice of $h$. So we can define $\widehat{f}(\mathbf{r})=g(1)$. $\widehat{f}(\mathbf{r})$ is a preimage of $f(p_G(\RQ(\mathbf{r})))$. By the construction of $\widehat{f}$, we can see that if $h$ is a path from $\mathbf{p}$ to $\mathbf{q}$, then the lift of $f\circ p_G\circ \RQ\circ h$ starting from $\widehat{f}(\mathbf{p})$ will end at $\widehat{f}(\mathbf{q})$.

Let $\widehat{G}=\RQ^{-1}[G]$ be the binary polyhedral group corresponding to $G$. Since $G$ contains a $180^\circ$ rotation, $\widehat{G}$ must contain a pure vector. Let $\mathbf{q}_0=x\mathbf{i}+y\mathbf{j}+z\mathbf{k}\in\widehat{G}$ be such a quaternion. Define $\mathbf{q}(t)=\cos \frac{\pi}{2} t + \sin \frac{\pi}{2} t\cdot \mathbf{q}_0$. Now for any $\mathbf{p}\in S^3$, let $r_\triangleleft(t)=\mathbf{p}\mathbf{q}(t)$ and $r_\triangleright(t)=\mathbf{p}\mathbf{q}_0\mathbf{q}(t)$, then $r_\triangleleft(t)$ is a path from $\mathbf{p}$ to $\mathbf{p}\mathbf{q}_0$, $r_\triangleright(t)$ is a path from $\mathbf{p}\mathbf{q}_0$ to $\mathbf{p}\mathbf{q}_0^2=-\mathbf{p}$ and $r_\triangleright(t)=r_\triangleleft(t)\mathbf{q}_0$

Note that $p_G(\RQ(\mathbf{p}))=p_G(\RQ(\mathbf{p}\mathbf{q}))$ if $\mathbf{q}\in\widehat{G}$. So, $p_G\circ \RQ\circ r_\triangleleft$ and $p_G\circ \RQ\circ r_\triangleright$ are the same path in $\so(3)/G$ and they are a loop. Then, $v=f\circ p_G\circ \RQ\circ r_\triangleleft=f\circ p_G\circ \RQ\circ r_\triangleright$ is a loop in $\so(3)$. Consider the lift of $v$ in $S^3$. Let $\mathbf{u}=\widehat{f}(\mathbf{p})$, then the lift of $v$ in $S^3$ must start and end at either $\mathbf{u}$ or $-\mathbf{u}$. If the lifted path starting from $\mathbf{u}$ also ends at $\mathbf{u}$, then the lifted path starting from $-\mathbf{u}$ must end at $-\mathbf{u}$, for otherwise the reverse of $v$ lifts to two paths path from $\mathbf{u}$ to two different destinations $\mathbf{u}$ and $-\mathbf{u}$ which is impossible. Likewise, if the lifted path starting from $\mathbf{u}$ ends at $-\mathbf{u}$, then the lifted path starting from $-\mathbf{u}$ must end at $\mathbf{u}$.

Now join $r_\triangleleft$ and $r_\triangleright$ as a long path $r$: $r(t)=\mathbf{q}(2t)$. $r$ is a path from $\mathbf{p}$ to $\mathbf{-p}$. The first half of $r$ coincides with $r_\triangleleft$ and the second half of $r$ coincides with $r_\triangleright$. Consider the lift of $f\circ p_G\circ \RQ\circ r$. It can be obtained by joining two lifts of $v$, lets call $\widehat{v}_\triangleleft$ and $\widehat{v}_\triangleright$. If $\widehat{v}_\triangleleft$ goes from $\mathbf{u}$ to $-\mathbf{u}$, then $\widehat{v}_\triangleright$ goes from $-\mathbf{u}$ to $\mathbf{u}$. If $\widehat{v}_\triangleleft$ loops from $\mathbf{u}$ back to $\mathbf{u}$, then $\widehat{v}_\triangleleft$ does the same. In any case, the whole lifted path is a loop from $\mathbf{u}$ back to $\mathbf{u}$.

The start of this path is $\mathbf{u}=\widehat{f}(\mathbf{p})$. By the construction of $\widehat{f}$, the end point equals $\widehat{f}(-\mathbf{p})$. So for any $\mathbf{p}\in S^3$, $\widehat{f}(\mathbf{p})=\widehat{f}(\mathbf{-p})$.

Recall the definition of distance from a rotation $S$ to an element $\mathcal{R}$ of $\so(3)/G$. We have
\begin{equation}
d_G(f(p_G(\RQ(\mathbf{p}))),p_G(\RQ(\mathbf{p})))=d_G(\mathbf{p},\widehat{f}(\mathbf{p}))=\min_{q\in\widehat{G}}2d_Q(\mathbf{p},\widehat{f}(\mathbf{p})\mathbf{q})
\end{equation}

Let us now select 3 quaternions $\mathbf{q}_1$, $\mathbf{q}_2$ and $\mathbf{q}_3$ from $\widehat{G}$. The method of selection will be considered later. Let $l_i(\mathbf{p})=\mathbf{p}\cdot\widehat{f}(\mathbf{p})-\mathbf{p}\cdot(\widehat{f}(\mathbf{p})\mathbf{q}_i)$, then
\begin{align}
    l_i(\mathbf{-p})&=(\mathbf{-p})\cdot\widehat{f}(\mathbf{-p})-(\mathbf{-p})(\cdot\widehat{f}(\mathbf{-p})\mathbf{q}_i) \nonumber \\
    &=-\mathbf{p}\cdot\widehat{f}(\mathbf{p})+\mathbf{p}\cdot(\widehat{f}(\mathbf{p})\mathbf{q}_i) \nonumber \\
    &=-l_i(\mathbf{p})
\end{align}

Let $L(\mathbf{p})=(l_1(\mathbf{p}),l_2(\mathbf{p}),l_3(\mathbf{p}))$. Similar to theorem \ref{thm:3funcs}, by the Borsuk–Ulam theorem, there exists $\mathbf{p}_0$ such that $L(\mathbf{p}_0)=\mathbf{0}$. Then we have
\begin{equation}
\label{eqn:dist_eqn}
    \mathbf{p}_0\cdot\widehat{f}(\mathbf{p}_0)-\mathbf{p}_0\cdot(\widehat{f}(\mathbf{p}_0)\mathbf{q}_1) = \mathbf{p}_0\cdot\widehat{f}(\mathbf{p}_0)-\mathbf{p}_0\cdot(\widehat{f}(\mathbf{p}_0)\mathbf{q}_2) = \mathbf{p}_0\cdot\widehat{f}(\mathbf{p}_0)-\mathbf{p}_0\cdot(\widehat{f}(\mathbf{p}_0)\mathbf{q}_3) = 0
\end{equation}
$\mathbf{a}\cdot(\mathbf{b}\mathbf{c})=\mathrm{Re}(\overline{\mathbf{a}}(\mathbf{b}\mathbf{c}))=\mathrm{Re}((\overline{\mathbf{a}}\mathbf{b})\mathbf{c})=(\overline{\mathbf{b}}\mathbf{a})\cdot\mathbf{c}$ for any quaternions $\mathbf{a}$, $\mathbf{b}$ and $\mathbf{c}$, so equivalently
\begin{equation}
    (\overline{\widehat{f}(\mathbf{p}_0)}\mathbf{p}_0)\cdot 1
    =(\overline{\widehat{f}(\mathbf{p}_0)}\mathbf{p}_0)\cdot \mathbf{q}_1
    =(\overline{\widehat{f}(\mathbf{p}_0)}\mathbf{p}_0)\cdot \mathbf{q}_2
    =(\overline{\widehat{f}(\mathbf{p}_0)}\mathbf{p}_0)\cdot \mathbf{q}_3
\end{equation}
In non-degenerate cases, this will give us three independent linear equations in the four components of $\overline{\widehat{f}(\mathbf{p}_0)}\mathbf{p}_0$. Together with $||\overline{\widehat{f}(\mathbf{p}_0)}\mathbf{p}_0||=1$, $\overline{\widehat{f}(\mathbf{p}_0)}\mathbf{p}_0$ can be uniquely determined up to negation. From this, we can compute
\begin{align}
d_G(\mathbf{p}_0,\widehat{f}(\mathbf{p}_0))&=\min_{q\in\widehat{G}}2d_Q(\mathbf{p}_0,\widehat{f}(\mathbf{p}_0)\mathbf{q})\nonumber\\
&=\min_{q\in\widehat{G}}2\cos^{-1}(\mathbf{p}_0\cdot(\widehat{f}(\mathbf{p}_0)\mathbf{q}))\nonumber\\
&=\min_{q\in\widehat{G}}2\cos^{-1}((\overline{\widehat{f}(\mathbf{p}_0)}\mathbf{p}_0)\cdot\mathbf{q})
\end{align}
which gives us a lower bound for $\max_{\mathbf{p}}d_G(\mathbf{p},\widehat{f}(\mathbf{p}))$, i.e. the lower bound for the maximum error that must occur. It turns out that the highest lower bound so obtained is exactly the largest possible value of $d_G(\mathbf{p},\mathbf{q})$ and so must be the best possible lower bound. The choice of $\mathbf{q}_i$'s that achieves this is such that $\overline{\widehat{f}(\mathbf{p}_0)}\mathbf{p}_0$ as solved from equation \ref{eqn:dist_eqn} achieves the largest value. 

We list all choices of $G$, along with the elements of $\widehat{G}$, one of many best choices of $\mathbf{q}_i$'s and the corresponding $\overline{\widehat{f}(\mathbf{p}_0)}\mathbf{p}_0$ and $d_G(\mathbf{p}_0,\widehat{f}(\mathbf{p}_0))$ in table \ref{tab:sym-bound}. In the table, $\mathrm{Perm}()$ means all permutations of the number sequence, excluding duplicate sequences, $\mathrm{EvenPerm}()$ is the same but with only even permutations. $C_n$ is a degenerate case where one of the $\mathbf{q}_i$ is redundant and the solution is not unique, but they all give the same bound. Our derivation does not work for $C_n$ when $n$ is odd, but we believe that the conclusion is the same as for $C_n$ where $n$ is even, namely, that the lower bound of highest error is $\pi$.

\renewcommand{\arraystretch}{1.5}
\begin{table}
\begin{floatrow}
\ttabbox{
    \small
    \centering
    \begin{tabular}{|c|c|c|c|c|}\hline
        \multirow{2}{*}{$G$} & \multirow{2}{*}{$|G|$} & \multirow{2}{*}{Elements of $\widehat{G}$} & \multirow{2}{*}{Choice of $\mathbf{q}_i$'s} & \rule{0pt}{12pt}$\overline{\widehat{f}(\mathbf{p}_0)}\mathbf{p}_0$ \\\cline{5-5} & & & & $d_G(\mathbf{p}_0,\widehat{f}(\mathbf{p}_0))$ \\\hline
        
        \multirow{2}{*}{$C_n$} & \multirow{2}{*}{$n$} &
        \multirow{2}{145pt}{\centering\vspace{2pt}
            $(\cos\frac{k\pi}{n},0,0,\sin\frac{k\pi}{n}), 0\le k < 2n$
        } & 
        \multirow{2}{80pt}{\centering\vspace{2pt}
            \hspace{-3pt}$(-1, 0, 0, 0)$\\\vspace{2pt}
            \hspace{3pt}$(\cos\frac{\pi}{n},0,0,\sin\frac{\pi}{n})$
        } &
        $(0, \cos\theta, \sin\theta, 0), 0\hspace{-2pt}\le\hspace{-2pt}\theta\hspace{-2pt}<\hspace{-2pt}2\pi$ \\\cline{5-5} & & & & $\pi$ \\\hline
        
        \multirow{4}{*}{$D_n$} & \multirow{4}{*}{$2n$} &
        \multirow{4}{145pt}{\centering\vspace{2pt}
            \hspace{-2.5pt}$(\cos\frac{k\pi}{n},0,0,\sin\frac{k\pi}{n}), 0\le k < 2n$\\\vspace{2pt}
            \hspace{2.5pt}$(0,\cos\frac{k\pi}{n},\sin\frac{k\pi}{n},0), 0\le k < 2n$
        } & 
        \multirow{4}{80pt}{\centering\vspace{2pt}
            \hspace{-3pt}$(\cos\frac{\pi}{n},0,0,\sin\frac{\pi}{n})$\\\vspace{2pt}
            \hspace{0pt}$(0,1,0,0)$\\\vspace{2pt}
            \hspace{3pt}$(0,\cos\frac{\pi}{n},\sin\frac{\pi}{n},0)$
        } &
        \multirow{2}{*}{{\tiny $\frac{\sqrt{2}}{2}(\cos\hspace{-2pt}\frac{\pi}{2n}\hspace{-1pt},\cos\hspace{-2pt}\frac{\pi}{2n}\hspace{-1pt},\sin\hspace{-2pt}\frac{\pi}{2n}\hspace{-1pt},\sin\hspace{-2pt}\frac{\pi}{2n}\hspace{-1pt})$}} \\ & & & & \\\cline{5-5} & & & & \multirow{2}{*}{$\cos^{-1}(-\sin^2\frac{\pi}{2n})$} \\ & & & & \\\hline
        
        \multirow{4}{*}{$T$} & \multirow{4}{*}{$12$} &
        \multirow{4}{145pt}{\centering\vspace{2pt}
            \hspace{-3pt}$\mathrm{Perm}(\pm1,0,0,0)$\\\vspace{2pt}
            \hspace{3pt}$(\pm\frac{1}{2},\pm\frac{1}{2},\pm\frac{1}{2},\pm\frac{1}{2})$
        } & 
        \multirow{4}{80pt}{\centering\vspace{2pt}
            \hspace{-3pt}$(\frac{1}{2},\frac{1}{2},\frac{1}{2},\frac{1}{2})$\\\vspace{2pt}
            \hspace{0pt}$(\frac{1}{2},\frac{1}{2},\frac{1}{2},-\frac{1}{2})$\\\vspace{2pt}
            \hspace{3pt}$(\frac{1}{2},\frac{1}{2},-\frac{1}{2},\frac{1}{2})$
        } &
        \multirow{2}{*}{$(\frac{\sqrt{2}}{2},\frac{\sqrt{2}}{2},0,0)$} \\ & & & & \\\cline{5-5} & & & & \multirow{2}{*}{$\frac{\pi}{2}$} \\ & & & & \\\hline
        
        \multirow{4}{*}{$O$} & \multirow{4}{*}{$24$} &
        \multirow{4}{145pt}{\centering\vspace{2pt}
            \hspace{-3pt}$\mathrm{Perm}(\pm1,0,0,0)$\\\vspace{2pt}
            \hspace{0pt}$(\pm\frac{1}{2},\pm\frac{1}{2},\pm\frac{1}{2},\pm\frac{1}{2})$\\\vspace{2pt}
            \hspace{3pt}$\mathrm{Perm}(\pm\frac{\sqrt{2}}{2},\pm\frac{\sqrt{2}}{2},0,0)$
        } & 
        \multirow{4}{80pt}{\centering\vspace{2pt}
            \hspace{-3pt}$(\frac{\sqrt{2}}{2},\frac{\sqrt{2}}{2},0,0)$\\\vspace{2pt}
            \hspace{0pt}$(\frac{\sqrt{2}}{2},0,\frac{\sqrt{2}}{2},0)$\\\vspace{2pt}
            \hspace{3pt}$(\frac{1}{2},\frac{1}{2},\frac{1}{2},\frac{1}{2})$
        } &
        \multirow{2}{*}{$(\frac{2+\sqrt{2}}{4},\frac{\sqrt{2}}{4},\frac{\sqrt{2}}{4},\frac{2-\sqrt{2}}{4})$} \\ & & & & \\\cline{5-5} & & & & \multirow{2}{*}{$\cos^{-1}\frac{2\sqrt{2}-1}{4}$} \\ & & & & \\\hline
        
        \multirow{4}{*}{$I$} & \multirow{4}{*}{$60$} &
        \multirow{4}{145pt}{\centering\vspace{2pt}
            \hspace{-3pt}$\mathrm{Perm}(\pm1,0,0,0)$\\\vspace{2pt}
            \hspace{0pt}$(\pm\frac{1}{2},\pm\frac{1}{2},\pm\frac{1}{2},\pm\frac{1}{2})$\\\vspace{2pt}
            \hspace{3pt}$\mathrm{EvenPerm}(0,\pm\frac{\sqrt{5}+1}{4},\pm\frac{1}{2},\pm\frac{\sqrt{5}-1}{4})$
        } & 
        \multirow{4}{80pt}{\centering\vspace{2pt}
            \hspace{-3pt}$(\frac{\sqrt{5}+1}{4},\frac{1}{2},0,\frac{\sqrt{5}-1}{4})$\\\vspace{2pt}
            \hspace{0pt}$(\frac{\sqrt{5}+1}{4},\frac{\sqrt{5}-1}{4},\frac{1}{2},0)$\\\vspace{2pt}
            \hspace{3pt}$(\frac{\sqrt{5}+1}{4},0,\frac{\sqrt{5}-1}{4},\frac{1}{2})$
        } &
        \multirow{2}{*}{ $\frac{\sqrt{10}-\sqrt{2}}{8}(\sqrt{5}+2,1,1,1)$} \\ & & & & \\\cline{5-5} & & & & \multirow{2}{*}{$\cos^{-1}\frac{3\sqrt{5}-1}{8}$} \\ & & & & \\\hline
    \end{tabular}
}
{
\caption{Table for computing error bounds of $f:\so(3)/G\to\so(3)$}
\label{tab:sym-bound}
}
\end{floatrow}
\end{table}
\renewcommand{\arraystretch}{1.0}

As a sanity check, now take the formula for $D_n$ and plug in $n=2$, we get
\begin{equation}
    \min_{f}\max_{\mathbf{p}}d_{D_2}(\mathbf{p},\widehat{f}(\mathbf{p}))=\cos^{-1}(-\sin^2\frac{\pi}{4})=\cos^{-1}(-\frac{1}{2})=\frac{2\pi}{3}
\end{equation}

which is indeed what the result in the second part of section \ref{sec:exp-2} suggested.

\begin{customprop}{\ref{thm:sym-four}}
\PropSymFourText
\end{customprop}
\textit{Remark.} We have seen in the proof of both theorem \ref{thm:1func} and theorem \ref{thm:sym} that the key technique is to find a loop in the base space that lifts to a non-loop in the covering space. The existence of such a loop is because the base space is not simply connected. If we instead only require $f:\so(3)/G\to\so(3)$ to satisfy $p_G(f(\mathcal{R}))=\mathcal{R}$ on a contractible 
open subset $U$ of $\so(3)/G$ then such $f$ could exist: a contractible open subset is always evenly covered, which means the preimage of $U$ under $p_G$ is the disjoint union of open subsets of $\so(3)$ each of which is homeomorphic to $U$ under $p_G$. We can select any one of these subsets $V$ and define $f$ to be the inverse of $p_G:V\to U$ on $U$, and assign values of $f$ on $\so(3)/G\setminus U$ so that $f$ is continuous on $\so(3)/G$.

The way to achieve this is exemplified in the proof of theorem \ref{thm:4funcs-e}: expand $U$ a bit to leave some margin. Beyond the margin, let $f$ be a constant function so that global structures of $\so(3)/G$ will not affect the continuity of $f$ outside this local patch. In theorem \ref{thm:4funcs-e}, this is done by setting $f$ to $(0,0,0)$ near gimbal locked positions and discontinuities of $\atan$. In the margin, let $f$ continuously change from being correct to being constant. In theorem \ref{thm:4funcs-e} this is done by blending $(\alpha_i(M),\beta_i(M),\gamma_i(M))$ and $(0,0,0)$ using $t_i(M)$.

Now if we can cover $\so(3)/G$ with finitely many such open contractible subsets, then we can construct an $f$ for each of them and these $f$'s will make a successful ensemble. For a topological space $X$, the smallest number $k$ such that there exists an open cover $\{U_i|1\le i \le k\}$ of $X$ with each $U_i$ contractible is called the Lusternik–Schnirelmann category of $X$, denoted $\mathrm{cat}(X)$. It is proved in \cite{takens1968minimal} that for a smooth compact manifold $X$, $\mathrm{cat}(X)\le\mathrm{dim}(X)+1$. $\so(3)/G$ is a 3-dimensional smooth compact manifold, so it can be covered by four contractible open subsets.

\section{Details of Experiments}
\label{sec:exp-more}

\subsection{Initialization of Last Layer}
In all networks, in both experiments and including both the conversion functions and the classifiers, the weight of the last layer is initialized to zero. For the classifiers, the bias of the last layer is initialized to $\frac{1}{n}$ for all output neurons where $n$ is the size of the ensemble. For the conversion functions, the bias of the last layer is initialized to the representation of a randomly sampled rotation. For the quaternion, 5D and 6D representations, the rotation is uniformly sampled from $\so(3)$. For the Euler angle, we want to avoid gimbal lock in the initial bias, so instead of uniform sampling from $\so(3)$, we uniformly sample $\alpha$ and $\gamma$ from $[-\pi,\pi]$ and $\beta$ from $[-\frac{\pi}{4},\frac{\pi}{4}]$.

\subsection{Penalizing Bad Representations}
We introduce a penalty that helps prevent bad representations. Although our theory considers output spaces of various forms, in practice the output of the neural network is in a Euclidean space and has to be normalized in some way to produce a valid rotation representation, for example, for quaternion the output must be normalized to have length $1$. This can be problematic if the unnormalized output is close to zero, since zero is a singularity and a small change in the unnormalized output near zero can result in a big change in the normalized output. So to improve the stability, we bound the output of the network away from singularities by adding a penalty on representations too far away from a valid rotation representation.

For quaternions, the penalty is $\mathcal{L}_{\text{p}}(\mathbf{p})=(\ln ||\mathbf{p}||)^2$. For a 6D representation $\mathbf{x}=(x_1, x_2, \ldots, x_6)$, let
\begin{equation}
    \mathbf{u}=(x_1,x_2,x_3)\quad\mathbf{v}=(x_4,x_5,x_6)\quad\mathcal{L}_{\text{p}}(\mathbf{x})=(\ln ||\mathbf{u}||)^2+(\ln ||\mathbf{v}||)^2+(\mathbf{u}\cdot\mathbf{v})^2
\end{equation}
For a 5D representation, first convert it to 6D by stereographic projection then apply the penalty for a 6D representation. The conversion from Euler angles to rotations is smooth, so no penalty is applied.

The 3D matrix conversion problem is simple enough that this penalty is found to be unnecessary. For the point cloud rotation estimation and 4D rotation matrix conversion that we will introduce in section \ref{sec:4d} the penalty is added to the loss function with an appropriate weight.

\subsection{Additional Results of Rotation Matrix Conversion}
Here we compare quaternion and Euler angle ensembles of different sizes to show that an ensemble of four is necessary. Figure \ref{fig:t1-quat} and table \ref{tab:t1-quat} compares the conversion error with a single quaternion and quaternion ensembles of size 2, 3 and 4. We can see that by introducing a second and a third network into the ensemble the conversion is reduced overall, but this does not prevent the occurrence of error close to $180^\circ$. In contrast, introducing a fourth network causes qualitative changes in that the maximum error is lowered drastically to only about $0.1^\circ$.

Figure \ref{fig:t1-euler} and table \ref{tab:t1-euler} compares the conversion error with a single set of Euler angles, uniform $x$-$y$-$z$ Euler angle ensembles of size 2, 3 and 4, and the mixed Euler angle ensemble with two set each of $x$-$y$-$z$ and $x$-$z$-$y$ Euler angles. Qualitatively we can observe the same behavior in that ensembles of size up to three give a maximum error of $180^\circ$ while the maximum error for an ensemble of four is much lower. The mixed ensemble further improves upon the uniform ensemble of four but this improvement is only quantitative.

Although qualitatively similar, overall Euler angles compare poorly to quaternions. We think that this is due to $\RQ$ being a local isometry. In a sense, the same amount of change anywhere in the input space will also cause the same amount of change in the output space, which is a favorable condition for the network to fit such a function. The same is true for the 6D embedding. Indeed, as seen in figure \ref{fig:t1-all}, the error curves of the 6D embedding and the ensemble of four quaternions almost match perfectly. In contrast, the Euler angle does not have this property. Near the gimbal locked positions, small changes in the input can cause huge changes in the output. The 5D embedding does not have this property either, as the stereographic projection does not preserve distance.

\begin{figure}[t]\CenterFloatBoxes
\begin{floatrow}
\ffigbox[1.1\linewidth]{
  \includegraphics[width=\linewidth]{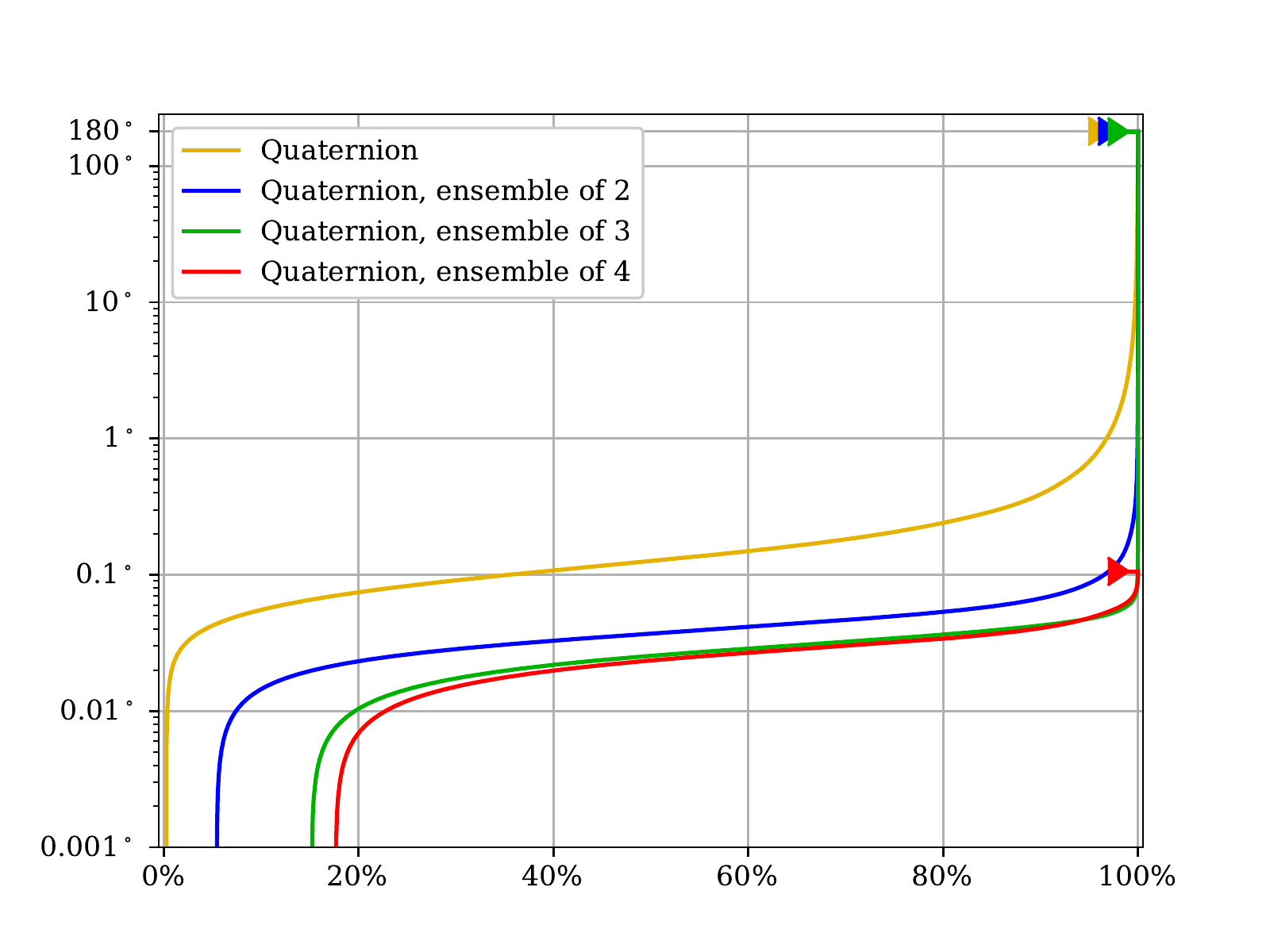}
}
{
\caption{Error of rotation matrix conversion by percentile, using different quaternion ensembles}
\label{fig:t1-quat}
}
\killfloatstyle\ttabbox{
{\footnotesize
\begin{tabular}{lrr}
  \toprule
  \cb{white} Type            & Mean($^\circ$) & Max($^\circ$)  \\
  \midrule
  \cb{clr0} Quat.            & $0.3323$       & $179.9995$ \\
  \cb{clr1} Quat. $\times$2  & $0.0443$       & $179.7149$ \\
  \cb{clr2} Quat. $\times$3  & $0.0243$       & $178.6680$ \\
  \cb{clr3} Quat. $\times$4  & $0.0226$       &   $0.1059$ \\
  \bottomrule
\end{tabular}
}
}
{
\caption{Error statistics}
\label{tab:t1-quat}
}
\end{floatrow}
\end{figure}
\vspace{15pt}

\begin{figure}[t]\CenterFloatBoxes
\begin{floatrow}
\ffigbox[1.1\linewidth]{
  \includegraphics[width=\linewidth]{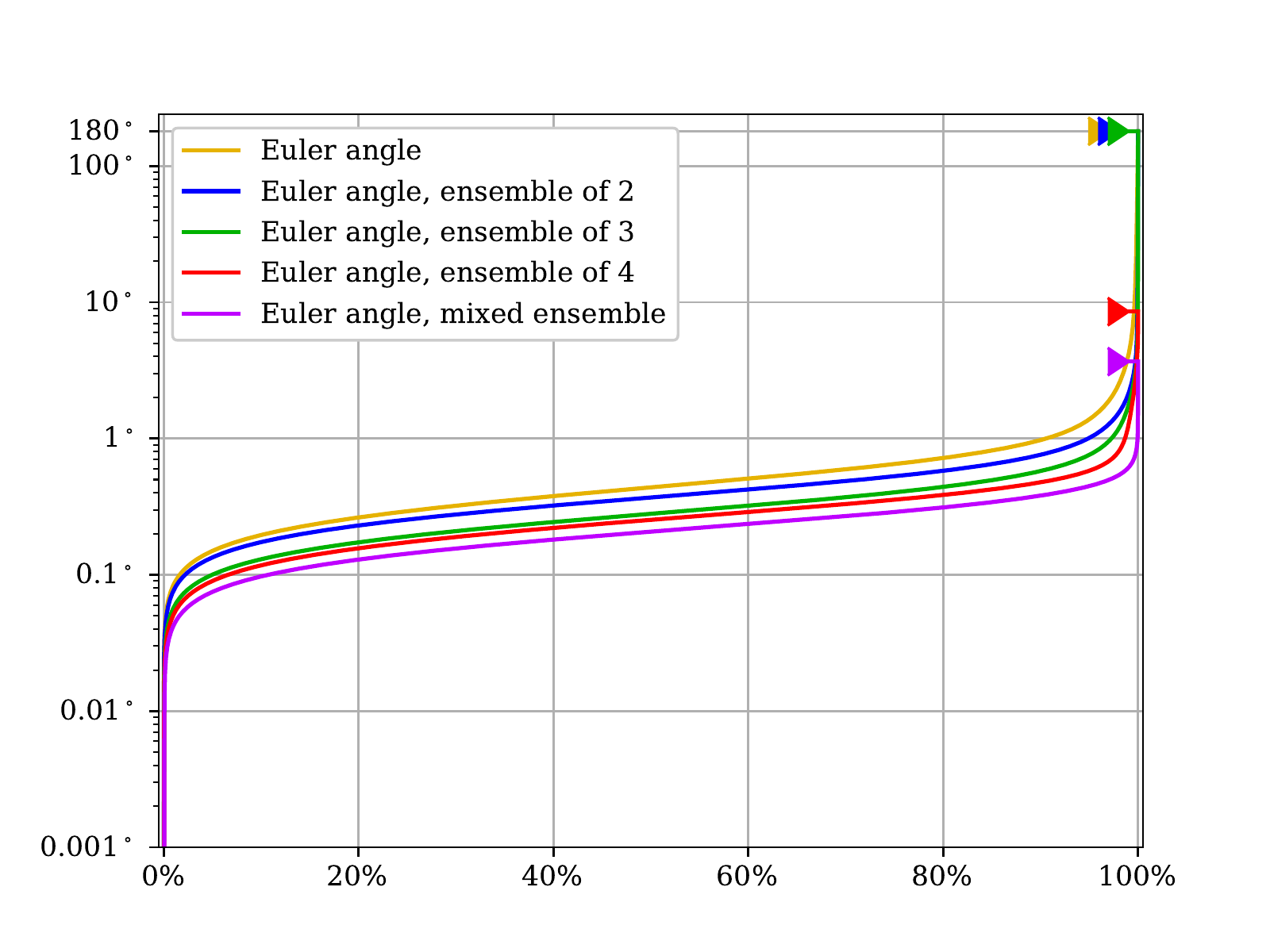}
}
{
\caption{Error of rotation matrix conversion by percentile, using different Euler angle ensembles}
\label{fig:t1-euler}
}
\killfloatstyle\ttabbox{
{\footnotesize
\begin{tabular}{lrr}
  \toprule
  \cb{white} Type            & Mean($^\circ$) & Max($^\circ$)  \\
  \midrule
  \cb{clr0} Euler            & $0.7368$       & $179.9991$ \\
  \cb{clr1} Euler $\times$2  & $0.4898$       & $179.9981$ \\
  \cb{clr2} Euler $\times$3  & $0.3578$       & $179.9979$ \\
  \cb{clr3} Euler $\times$4  & $0.2990$       &   $8.5356$ \\
  \cb{clr4} Euler mix        & $0.2278$       &   $3.6693$ \\
  \bottomrule
\end{tabular}
}
}
{
\caption{Error statistics}
\label{tab:t1-euler}
}
\end{floatrow}
\end{figure}

\subsection{Setup of Point Cloud Rotation Estimation}

Our network is a modified PointNet \cite{qi2017pointnet}. In the full PointNet there are two feature transforms on local features before max-pooling is applied to obtain global features. The transforms takes the form of a matrix multiplication, and the matrices are computed from a T-Net, which is a simplified PointNet, thus resulting in a nested structure. We consider this to be too complex for our simple problem, but we also think that this causes some global information to be combined into local features before the max pooling, which might be important. So instead of a complete removal of these nested networks, we replace it with a simpler way for incorporating global information: We perform max pooling, multiply the pooled feature with a learned matrix and apply ReLU, then concatenate the resulting feature vector to the local feature of each point. This doubles the number of features at each point, so the size of the input layer of the subsequent MLP in PointNet is modified accordingly.

We stress however that our discussion is focused on the topological relationship between the input space and the output space, rather than concrete aspects of implementation, including the choice of network architecture. As long as the network is invariant under input point permutation and has appropriate capacity for our example problem, the same qualitative result should be expected.

As mentioned in section \ref{sec:exp-2}, in each part we only use a single point cloud, for both training and testing. Normally to demonstrate the effectiveness of a machine learning algorithm, the model is trained on a training dataset of suitable size and tested on a separate test dataset, and our setup seems rather unconventional. But note that our purpose is not to demonstrate the effectiveness of a machine learning algorithm. Rather, we aim to examine a property of a machine learning problem itself. We want to show that the unavoidable large error is a result of the topological structure of the problem.

If the conventional setting is used and a large maximum arises in the test, the source of error can be hard to explain. It might be because the network failed to generalize. It might be because the network does not have enough capacity to give low error on every sample from the dataset. It might be because the topological property of the rotation representation guarantees that large errors must occur. The last one is what we want to show. The best way to prove that a large error is indeed an inherent property of the representation is to avoid introducing any possible error due to the first two reasons at all. In fact, for a moderately difficult problem, it is almost never possible to train a neural network that generalizes perfectly. Since the existence or nonexistence of large error is the main differentiator, and since even a single instance of generalization failure can result in networks using different representations giving the same largest error, it is almost ensured that we will not be able to draw any useful conclusions if a conventional training/test split is used.

Consider a negative example. In the same point cloud rotation estimation experiment in \cite{zhou2019continuity}, a large dataset of thousands of point clouds was used, with a training/test split. Theoretically, since their 5D and 6D representations are continuous and their dataset of point clouds of aeroplanes can be safely assumed to possess no rotational symmetry, a sufficiently good network should always give small error, while for the quaternion and Euler angle representation large errors are guaranteed to occur. But as shown in figure 5f in \cite{zhou2019continuity}, their experiment failed to 
reflect this important qualitative difference between continuous and non-continuous representations since every representation gave a largest error close to $180^\circ$, which presumably was due to generalization error. We need to avoid this. Thus, we adopted the setting that only the error on the training set should be considered, and the training set should comprise only one point cloud, with its rotation being the only variable quantity.

In short, certain rotation representations are bad precisely because they are guaranteed to produce a large error, even on a training dataset with only one sample.

We then introduce how we constructed our point cloud data with the desired symmetry. The basis of our point cloud is a down-sampled version of the Stanford bunny. We first normalize the model by scaling and translation so that the bounding sphere of its axis-aligned bounding box is the unit sphere. Then a 3D model consisting of four Stanford bunnies arranged to have $D_2$ symmetry is constructed by taking four copies of the base model, transformed by the following four matrices respectively:
\begin{equation}
    \begin{bmatrix}1&0&0&0.5\\0&1&0&0.5\\0&0&1&0.5\end{bmatrix}
    \begin{bmatrix}1&0&0&0.5\\0&-1&0&-0.5\\0&0&-1&-0.5\end{bmatrix}
    \begin{bmatrix}-1&0&0&-0.5\\0&1&0&0.5\\0&0&-1&-0.5\end{bmatrix}
    \begin{bmatrix}-1&0&0&-0.5\\0&-1&0&-0.5\\0&0&1&0.5\end{bmatrix}
\end{equation}
Our point cloud with $D_2$ symmetry is the set of vertices of the resulting 3D model.

The bunny itself has no rotational symmetry. But we desired a point cloud with no symmetry that is otherwise similar to the point cloud with $D_2$ symmetry, so we took four copies the base model and translated them by $(0.5, 0.5, 0.5)$, $(0.5, -0.5, -0.5)$, $(-0.5, 0.5, -0.5)$ and $(-0.5, -0.5, 0.5)$ and took the vertices to form our point cloud with no symmetry.

\begin{figure}
\centering
\begin{subfigure}{0.4\textwidth}
  \includegraphics[width=\textwidth]{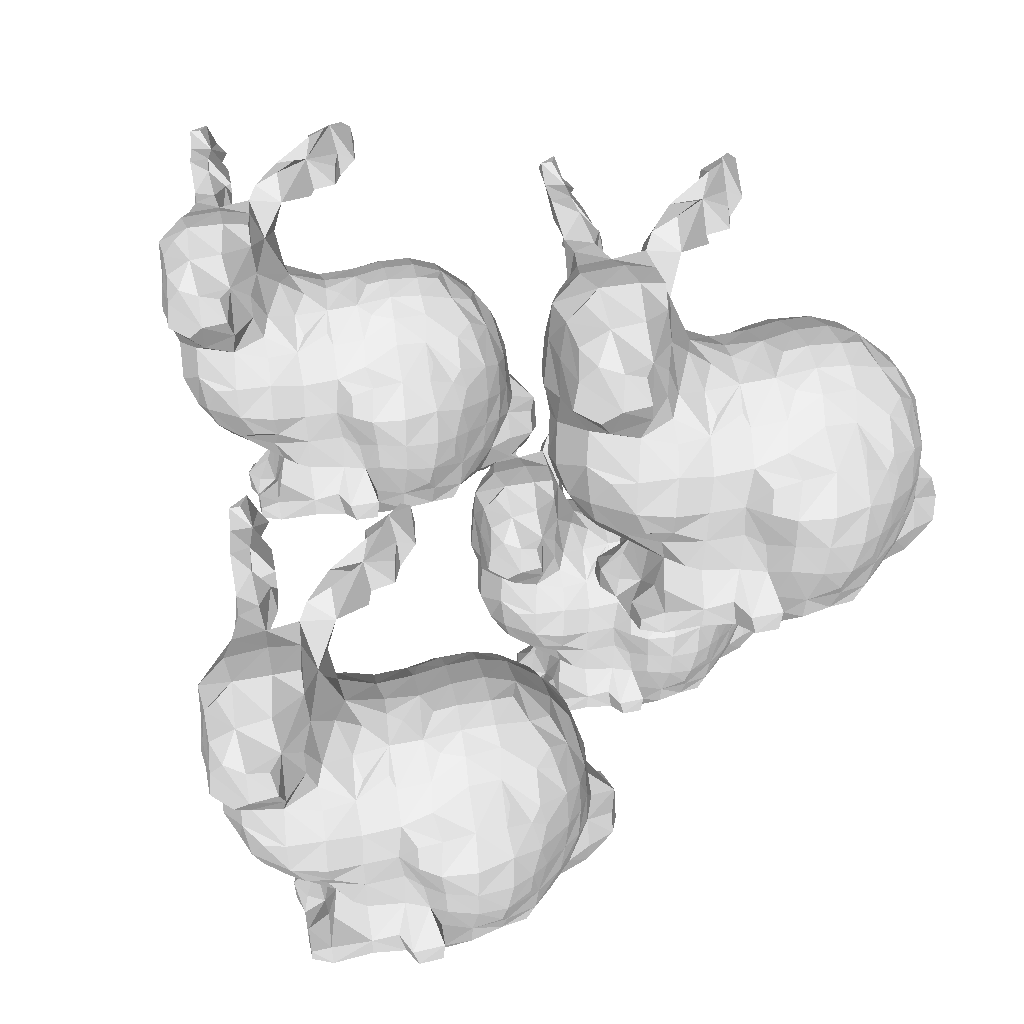}
  \caption{With no symmetry}
\end{subfigure}
\begin{subfigure}{0.4\textwidth}
  \includegraphics[width=\textwidth]{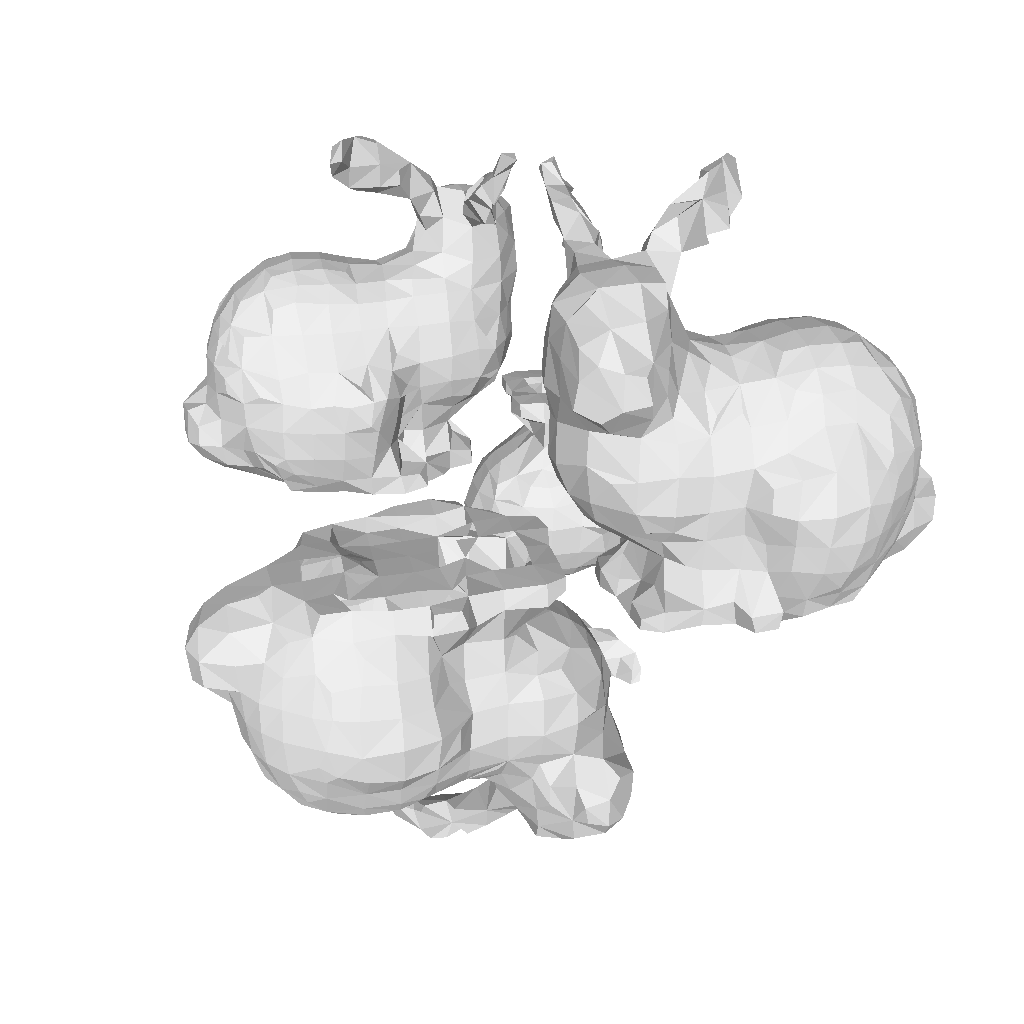}
  \caption{With $D_2$ symmetry}
\end{subfigure}
\caption{3D models of quadruple Stanford bunnies with different symmetries}
\label{fig:pointcloud}
\end{figure}

The resulting 3D models are shown in figure \ref{fig:pointcloud}. Since point clouds of scanned symmetric real-world objects typically do not have perfect symmetry, we simulate this effect by adding Gaussian noise to point coordinates. The displacement of each point is independent, and generated independently on the fly for each training sample.

We train each network using Adam with learning rate $10^{-4}$ and batch size $8$ for $500,000$ iterations. For training and testing, we sample uniformly from $\so(3)$. We sample $3$ million random rotations for testing. To improve stability, the Gaussian noise is introduced gradually. In the first $100$k iterations, We do not add any noise. Then for the next $100$k iterations, the standard deviation of the Gaussian noise increases linearly from $0$ to $0.01$. Then for the rest of the training it stays at $0.01$. The representation penalty is imposed with a weight of $1$ for the first $300$k iterations. Then for the next $100$k iterations the weight linearly decreases to $0$. For the rest of the training the penalty is removed.

The loss function is the same as equation \ref{eqn:loss} but with $d()$ replaced with $d_G()$. $\mathcal{L}_{\text{max}}$ may also be altered to take the maximum value of $d_G()$ but the difference is minor as long as it is no smaller than the actual maximum possible error so for simplicity we keep it at $\pi$.

\section{Extension to 4D Rotations}
\label{sec:4d}
In this section we extend some of our results to the quaternion pair representation of 4D rotations.

\subsection{Theoretical Results}
We first prove the error bound of conversion from 4D rotation matrices to quaternion pairs. We jump straight to an ensemble of three functions, since it implies the same result for a single function.
\begin{thm}
For any three continuous functions $f_1, f_2, f_3: \so(4) \to S^3 \times S^3$, there exists a rotation $R\in \so(4)$ such that $d_4(R,f_i(R))\ge\pi$ for all $i\in\{1, 2, 3\}$.
\end{thm}
\begin{proof}
Define the dot product between pairs of quaternions as $(\mathbf{p},\mathbf{q})\cdot(\mathbf{r},\mathbf{s})=\mathbf{p}\cdot\mathbf{q}+\mathbf{r}\cdot\mathbf{s}$. Consider functions $v_i: S^3 \to \mathbb{R}$ defined by $v_i(\mathbf{q})=(\mathbf{q},\mathbf{q})\cdot f_i(\RQQ(\mathbf{q},\mathbf{q}))$, for $i=1, 2, 3$. For any
$\mathbf{q} \in S^3$, Since
$\RQQ(-\mathbf{q},-\mathbf{q})=\RQQ(\mathbf{q},\mathbf{q})$, 
\begin{equation}
    v_i(-\mathbf{q})=(-\mathbf{q},-\mathbf{q})\cdot f_i(\RQQ(-\mathbf{q},-\mathbf{q}))=-(\mathbf{q},\mathbf{q})\cdot f_i(\RQQ(\mathbf{q},\mathbf{q}))=-v_i(\mathbf{q})
\end{equation}

Let $V:S^3 \to \mathbb{R}^3$ defined by $V(\mathbf{q})=(v_1(\mathbf{q}), v_2(\mathbf{q}), v_3(\mathbf{q}))$, then $V(-\mathbf{q})=-V(\mathbf{q})$
for any $\mathbf{q} \in S^3$.
By the Borsuk–Ulam theorem, there exists $\mathbf{q}_0 \in S^3$ such that $V(-\mathbf{q}_0)=V(\mathbf{q}_0)$, then $-V(\mathbf{q}_0)=V(-\mathbf{q}_0)=V(\mathbf{q}_0)$,
so $V(\mathbf{q}_0)=\mathbf{0}$, which means $v_1(\mathbf{q}_0)=v_2(\mathbf{q}_0)=v_3(\mathbf{q}_0)=0$.

Assume that $f_i(\RQQ(\mathbf{q}_0,\mathbf{q}_0))=(\mathbf{r}_i,\mathbf{s}_i)$. Then $(\mathbf{q}_0,\mathbf{q}_0)\cdot(\mathbf{r}_i,\mathbf{s}_i)=0$ so $\mathbf{q}_0\cdot\mathbf{r}_i=-\mathbf{q}_0\cdot\mathbf{s}_i$. So,
\begin{align}
    d_Q(\mathbf{q}_0,\mathbf{r}_i)&=\cos^{-1}(\mathbf{q}_0\cdot\mathbf{r}_i)\nonumber\\
    &=\cos^{-1}(-\mathbf{q}_0\cdot\mathbf{s}_i)\nonumber\\
    &=\pi-\cos^{-1}(\mathbf{q}_0\cdot\mathbf{s}_i)\nonumber\\
    &=\pi-d_Q(\mathbf{q}_0,\mathbf{s}_i)
\end{align}
So,
\begin{align}
    &d_4((\mathbf{q}_0,\mathbf{q}_0),(\mathbf{r}_i,\mathbf{s}_i))\nonumber\\
    =&\min\{d_Q(\mathbf{q}_0,\mathbf{r}_i)+d_Q(\mathbf{q}_0,\mathbf{s}_i),2\pi-d_Q(\mathbf{q}_0,\mathbf{r}_i)-d_Q(\mathbf{q}_0,\mathbf{s}_i)\}\nonumber\\
    &+|d_Q(\mathbf{q}_0,\mathbf{r}_i)-d_Q(\mathbf{q}_0,\mathbf{s}_i)|\nonumber\\
    \ge&\min\{\pi,2\pi-\pi\}\nonumber\\
    =&\pi
\end{align}
So, $\RQQ(\mathbf{q}_0,\mathbf{q}_0)$ is a rotation such that $d_4(\RQQ(\mathbf{q}_0,\mathbf{q}_0),f_i(\RQQ(\mathbf{q}_0,\mathbf{q}_0)))\ge\pi$ for all $i\in\{1, 2, 3\}$.
\end{proof}

In 3D, our lower bound of maximum error, $\pi$, is also the maximum value of $d(\mathbf{p},\mathbf{q})$, so it must be optimal. Here the maximum value of $d_4((\mathbf{p},\mathbf{q}),(\mathbf{r},\mathbf{s}))$ is $2\pi$ but we only proved a lower bound of $\pi$. We show however that in this case the lower bound is also optimal.

\begin{thm}
\label{thm:4d-tight-bound}
There exists a continuous function $f:\so(4)\to S^3\times S^3$ such that for any rotation $R\in\so(4)$, $d_4(R,f(R))\le\pi$.
\end{thm}
\begin{proof}
We give an example of such a function. For clarity we define $f$ on the quaternion representation. If we ensure that $f(\mathbf{p},\mathbf{q})=f(-\mathbf{p},-\mathbf{q})$ for all $\mathbf{p},\mathbf{q}\in S^3$, $f$ will be a well defined function of rotations.

Let $f(\mathbf{p},\mathbf{q})=(\mathbf{1},\mathbf{p}\mathbf{q})$, then for any $\mathbf{p},\mathbf{q}\in S^3$, $f(-\mathbf{p},-\mathbf{q})=(\mathbf{1},(-\mathbf{p})(-\mathbf{q}))=(\mathbf{1},\mathbf{p}\mathbf{q})=f(\mathbf{p},\mathbf{q})$, and
\begin{align}
    &d_4((\mathbf{p},\mathbf{q}),f(\mathbf{p},\mathbf{q}))\nonumber\\
    =&d_4((\mathbf{p},\mathbf{q}),(\mathbf{1},\mathbf{p}\mathbf{q}))\nonumber\\
    =&\min\{d_Q(\mathbf{p},\mathbf{1})+d_Q(\mathbf{q},\mathbf{p}\mathbf{q}),2\pi-d_Q(\mathbf{p},\mathbf{1})-d_Q(\mathbf{q},\mathbf{p}\mathbf{q}))\}+|d_Q(\mathbf{p},\mathbf{1})-d_Q(\mathbf{q},\mathbf{p}\mathbf{q}))|\nonumber\\
    \le&\pi+|\cos^{-1}(\mathrm{Re}(\mathbf{p}))-\cos^{-1}(\mathrm{Re}(\mathbf{p}))|\nonumber\\
    =&\pi
\end{align}
\end{proof}

Similar to the 3D case, we can construct a successful ensemble of four functions. In fact, the construction for 4D directly uses the construction for 3D.

\begin{thm}
\label{thm:4d-4funcs}
There exists continuous functions $g_1, g_2, g_3, g_4: \so(4) \to S^3 \times S^3$ such that for any rotation $R\in \so(4)$, $\RQQ(g_i(R))=R$ for some $i\in\{1, 2, 3, 4\}$.
\end{thm}
\begin{proof}
We give an example of such a set of functions. For clarity, we define them as functions of quaternions. If we ensure that $g_i(\mathbf{p},\mathbf{q})=g_i(-\mathbf{p},-\mathbf{q})$ for all $\mathbf{p},\mathbf{q}\in S^3$ and $i\in\{1, 2, 3, 4\}$, they will be well-defined functions of rotations. Use the definition of $f_i$ from the proof of theorem \ref{thm:4funcs}. Let
\begin{equation}
   g_i(\mathbf{p},\mathbf{q})=(f_i(\mathbf{p}),\overline{f_i(\mathbf{p})}\mathbf{p}\mathbf{q})
\end{equation}
then
\begin{equation}
    g_i(-\mathbf{p},-\mathbf{q})=(f_i(-\mathbf{p}),\overline{f_i(-\mathbf{p})}(-\mathbf{p})(-\mathbf{q}))
    =(f_i(\mathbf{p}),\overline{f_i(\mathbf{p})}\mathbf{p}\mathbf{q}))=g_i(\mathbf{p},\mathbf{q})
\end{equation}
By theorem \ref{thm:4funcs}, for all $\mathbf{p}\in S^3$ there exists $k\in\{1,2,3,4\}$ such that $\RQ(\mathbf{p})=\RQ(f_k(\mathbf{p}))$, which means $f_k(\mathbf{p})=\mathbf{p}$ or $\mathbf{p}=-\mathbf{p}$. Now for any $\mathbf{p},\mathbf{q}\in S^3$ find the $k$ such that $\RQ(\mathbf{p})=\RQ(f_k(\mathbf{p}))$. If $f_k(\mathbf{p})=\mathbf{p}$, then $g_k(\mathbf{p},\mathbf{q})=(\mathbf{p},\overline{\mathbf{p}}\mathbf{p}\mathbf{q})=(\mathbf{p},\mathbf{q})$. If $f_k(\mathbf{p})=-\mathbf{p}$, then $g_k(\mathbf{p},\mathbf{q})=(-\mathbf{p},-\overline{\mathbf{p}}\mathbf{p}\mathbf{q})=(-\mathbf{p},-\mathbf{q})$. In either case, $\RQQ(\mathbf{p},\mathbf{q})=\RQQ(g_k(\mathbf{p},\mathbf{q}))$.
\end{proof}

\subsection{Experiments}
We test the accuracy of converting 4D rotation matrices to pairs of quaternions using ensembles of different sizes. The size of the MLP used for this experiment is increased to $8$ hidden layers of size $256$ each. training time, batch size and learning rate remain he same as section \ref{sec:exp-1}. The result is shown in figure \ref{fig:t1-4d} and table \ref{fig:t1-4d}.

We found that, in terms of eliminating large errors, the conversion of 4D rotation matrices is considerably harder than 3D rotation matrices. In particular, the existence or nonexistence of a large maximum error is the main differentiator between ensembles of three or fewer networks and ensembles with four or more networks and we aim to show this difference, but stochastic gradient descent aims at reducing average loss with does not necessarily result in a low maximum error. Thus we modified the training procedure to put more emphasis on inputs with large error: in each iteration, all input samples in the training batch were sorted by their error. Half of the current batch with the largest errors were retained in the training batch of the next iteration, while the other half of the next training batch was randomly sampled.

Even with this modification, we were not able to reliably train an ensemble of four networks that result in a low maximum error. So we additionally tested with an ensemble of five networks, which when combined with this training strategy proved much easier to train than an ensemble of four. We think that this is likely due to the ``correct region'' of some or all networks in the ensemble of four being noncontractible. Note that there is a difference between the construction in the proof of theorem \ref{thm:4d-4funcs} compared to those in theorems \ref{thm:2d-2funcs}, \ref{thm:4funcs} and \ref{thm:4funcs-e}: in the earlier theorems, the correct region of each given function is contractible, while in theorem \ref{thm:4d-4funcs}, the correct region of each function is homeomorphic to $D^3\times S^3$ ($D^n$ is the closed disk) which is homotopic to $S^3$ and thus noncontractible. The classifier may have difficulty dividing $\so(4)$ into noncontractible regions without strong supervision.

Indeed, it is not possible to cover $\so(4)$ with $4$ contractible open subsets, as it is known that $\mathrm{cat}(\so(4))=5$. Possibly as a result of this, we can see that we succeeded with an ensemble of five.

Nevertheless, we show that a successful ensemble of four networks do exist by forcing it to learn the functions we constructed in theorem \ref{thm:4d-4funcs}. We modify the loss function as follows: let $\mathcal{L}_i(R)$ be the loss of network $i$ on input matrix $R$, define
\begin{align}
    \mathcal{L}_1(\RQQ(\mathbf{p},\mathbf{q}))=&\begin{cases}
        d_4((\mathbf{p},\mathbf{q}), f_1(\RQQ(\mathbf{p},\mathbf{q}))) & (|\mathrm{Re}(\mathbf{p})|\ge\frac{1}{2}) \\
        2\pi & (|\mathrm{Re}(\mathbf{p})|<\frac{1}{2})
    \end{cases}\\
    \mathcal{L}_2(\RQQ(\mathbf{p},\mathbf{q}))=&\begin{cases}
        d_4((\mathbf{p},\mathbf{q}), f_2(\RQQ(\mathbf{p},\mathbf{q}))) & (|\mathrm{Re}(\mathbf{pi})|\ge\frac{1}{2}) \\
        2\pi & (|\mathrm{Re}(\mathbf{pi})|<\frac{1}{2})
    \end{cases}\nonumber\\
    \mathcal{L}_3(\RQQ(\mathbf{p},\mathbf{q}))=&\begin{cases}
        d_4((\mathbf{p},\mathbf{q}), f_3(\RQQ(\mathbf{p},\mathbf{q}))) & (|\mathrm{Re}(\mathbf{pj})|\ge\frac{1}{2}) \\
        2\pi & (|\mathrm{Re}(\mathbf{pj})|<\frac{1}{2})
    \end{cases}\nonumber\\
    \mathcal{L}_4(\RQQ(\mathbf{p},\mathbf{q}))=&\begin{cases}
        d_4((\mathbf{p},\mathbf{q}), f_4(\RQQ(\mathbf{p},\mathbf{q}))) & (|\mathrm{Re}(\mathbf{pk})|\ge\frac{1}{2}) \\
        2\pi & (|\mathrm{Re}(\mathbf{pk})|<\frac{1}{2})
    \end{cases}\nonumber
\end{align}
That is, the $\mathcal{L}_i$ takes the maximum value $2\pi$ regardless of the output if the input does not lie in the correct region of $f_i$ as defined in the proof of theorem \ref{thm:4d-4funcs}. The loss function of the whole ensemble is derived from these as in equation \ref{eqn:loss}. We train the networks with this modified loss function for the first $300$k iterations, then revert to the normal loss function for another $200$k iterations. The result is shown as ``Theorem 14''. We can see that this hand-designed ensemble of four achieves lower maximum error than the ensemble of five but higher average error than the ensemble of three, which shows again that lower average error and lower maximum error might be conflicting goals and if eliminating large errors is considered important then we may need more suitable training methods.

Another point of note is that all autonomous ensembles up to size $4$ gave a maximum error close to $360^\circ$. We would like to verify theorem \ref{thm:4d-tight-bound}, and again we are faced with the problem with lowering the maximum error. This time, we deliberately let the network learn the ``wrong'' function we constructed in the proof of theorem \ref{thm:4d-tight-bound}. That is, the loss function of the network used for training is $\mathcal{L}(\RQQ(\mathbf{p},\mathbf{q}))=d_4((\mathbf{p},\mathbf{q}),(\mathbf{1},\mathbf{p}\mathbf{q}))$ while for testing the error is measured as normal. The result is shown as ``Theorem 13''. While the average error is awfully large (in theory it 
equals $\frac{\pi}{2}+\frac{2}{\pi}\approx 126.4756^\circ$, the average angle of a uniform random 3D rotation), the maximum error is close to $180^\circ$.

\begin{figure}\CenterFloatBoxes
\begin{floatrow}
\ffigbox[1.1\linewidth]{
  \includegraphics[width=\linewidth]{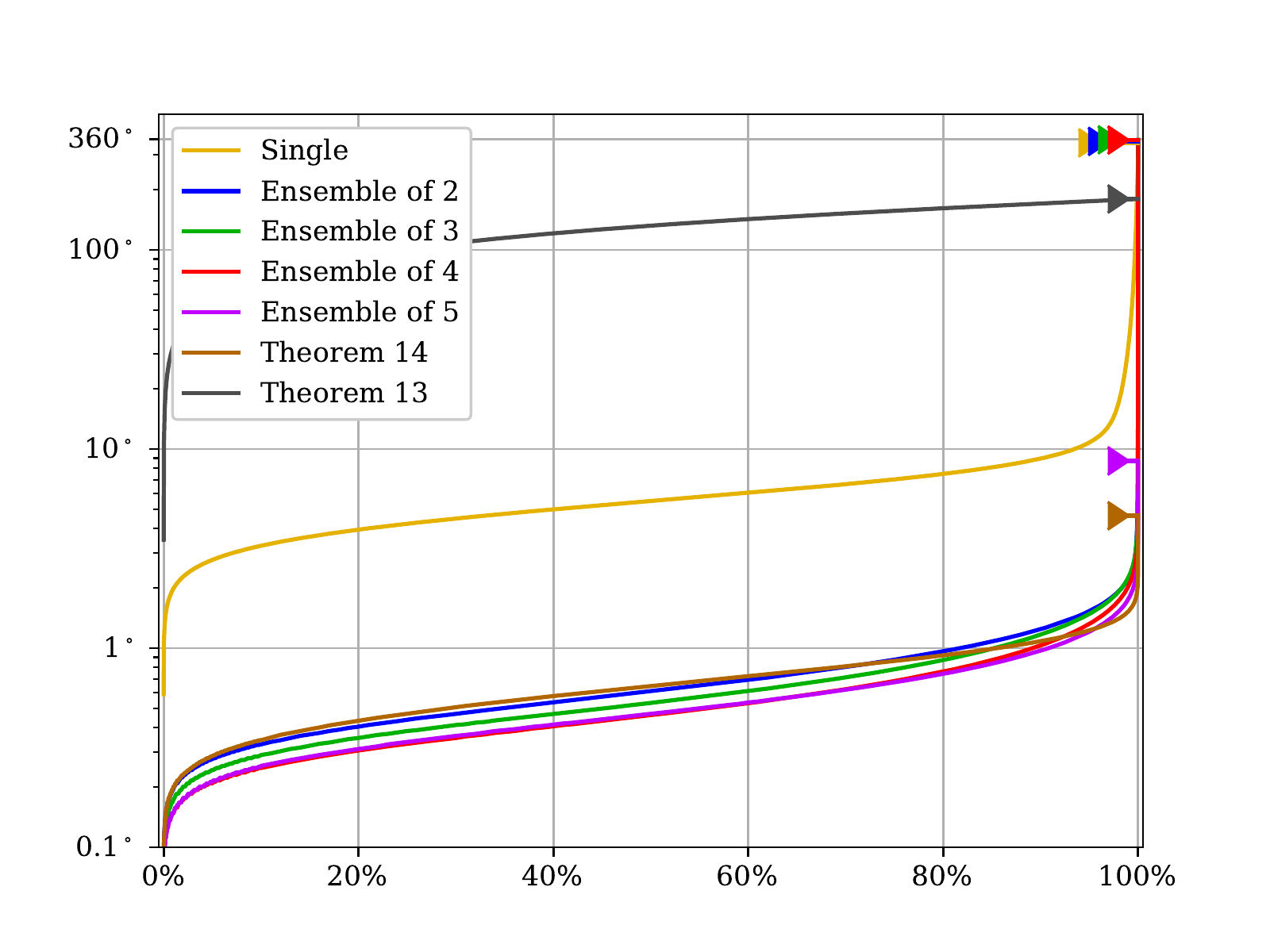}
}
{
\caption{Error of 4D rotation matrix conversion by percentile, using different quaternion pair ensembles}
\label{fig:t1-4d}
}
\killfloatstyle\ttabbox{
{\footnotesize
\begin{tabular}{lrr}
  \toprule
  \cb{white} Type     & Mean($^\circ$) & Max($^\circ$)  \\
  \midrule
  \cb{clr0} Single    &   $6.7637$     & $344.7241$ \\
  \cb{clr1} $\times$2 &   $0.7298$     & $350.5893$ \\
  \cb{clr2} $\times$3 &   $0.6702$     & $357.1087$ \\
  \cb{clr3} $\times$4 &   $0.5760$     & $356.0749$ \\
  \cb{clr4} $\times$5 &   $0.5563$     &   $8.7145$ \\
  \cb{clr5} Thm. 14   &   $0.6877$     &   $4.6233$ \\
  \cb{clr6} Thm. 13   & $126.5174$     & $180.1556$ \\
  \bottomrule
\end{tabular}
}
}
{
\caption{Error statistics}
\label{tab:t1-4d}
}
\end{floatrow}
\end{figure}

\end{document}